\documentclass[a4paper,twoside,reqno]{amsart}
\usepackage{amssymb}
\usepackage{amsfonts}
\usepackage{mathrsfs}
\usepackage{amsmath}
\usepackage{amsthm}
\usepackage{slashed}
\usepackage{url}
\usepackage{bm}
\usepackage{indentfirst}
\usepackage{upgreek}
\usepackage{hyperref}
\usepackage{graphicx}
\usepackage{setspace}
\usepackage{cite}
\usepackage{lmodern}
\usepackage[T1]{fontenc}
\usepackage{xcolor}
\usepackage{float}
\usepackage[all,cmtip]{xy}
\usepackage{mathtools}
\usepackage{caption}
\usepackage{tikz-cd}

\def\Hom{\mathop{\rm Hom}}

\newcommand{\nc}{\mathsf}
\newcommand{\ca}{\mathcal}

\newcounter{oftheorem}[section]
\newenvironment{mytheorem}[1]%
{\begin{trivlist}
		\renewcommand{\theoftheorem}{#1~\thesection.\arabic{oftheorem}}
		\refstepcounter{oftheorem}
		\item[\hspace{\labelsep}\bf\thesection.\arabic{oftheorem} #1.]}%
	{\end{trivlist}}
\newenvironment{definition}{\begin{mytheorem}{Definition}}{\end{mytheorem}}
\newenvironment{lemma}{\begin{mytheorem}{Lemma}\it}{\end{mytheorem}}
\newenvironment{proposition}{\begin{mytheorem}{Proposition}\it}{\end{mytheorem}}
\newenvironment{theorem}{\begin{mytheorem}{Theorem}\it}{\end{mytheorem}}
\newenvironment{corollary}{\begin{mytheorem}{Corollary}\it}{\end{mytheorem}}
\newenvironment{remark}{\begin{mytheorem}{Remark}}{\end{mytheorem}}
\newenvironment{example}{\begin{mytheorem}{Example}}{\end{mytheorem}}

\begin{document}
	
\title{The Exact Completion for Regular Categories enriched in Posets}		
\author{Vasileios Aravantinos-Sotiropoulos}

\begin{abstract}
	We construct an exact completion for regular categories enriched in the cartesian closed category $\mathsf{Pos}$ of partially ordered sets and monotone functions by employing a suitable calculus of relations. We then characterize the embedding of any regular category into its completion and use this to obtain examples of concrete categories which arise as such completions. In particular, we prove that the exact completion in this enriched sense of both the categories of Stone and Priestley spaces is the category of compact ordered spaces of L. Nachbin. Finally, we consider the relationship between the enriched exact completion and categories of internal posets in ordinary categories.
\end{abstract}

\maketitle

\section{Introduction}

The notions of regularity and (Barr-)exactness have been fundamental in Category Theory for quite some time. Exactness was introduced by Barr\cite{Barr} in 1970 and motivated by a result of Tierney which essentially exhibited the notion as the non-additive part of the definition of abelian category. From another perspective, it is the basic property which is common to both abelian categories and elementary toposes. Regularity is a weaker property which can be viewed as the requirement that the category affords a good calculus of internal relations. Alternatively, from the perspective of Categorical Logic, regular categories are those which correspond to the fragment of first-order Logic on the operations $\land,\top,\exists$.

In this paper we look at these notions in an enriched setting. More precisely, we work with versions of them that apply to categories enriched over the cartesian closed category $\nc{Pos}$ of partially ordered sets and monotone functions. Our main motivation comes from the paper \cite{Kurz-Velebil} by A. Kurz and J. Velebil, where $\nc{Pos}$-enriched regularity and exactness were first explicitly considered. The authors employ these notions to obtain categorical characterizations of (quasi-)varieties of ordered algebras in the sense of Bloom \& Wright\cite{Bloom & Wright}, very much along the lines of the corresponding characterizations for ordinary (quasi-)varieties of Universal Algebra. Broadly speaking, varieties turn out to be the exact categories possessing a ``nice'' generator, while quasivarieties can be characterized in a similar fashion by replacing exactness with the weaker regularity.

Recall here that \emph{ordered algebras} in the sense of \cite{Bloom & Wright} are algebras over some signature $\Sigma$ which consist of a poset $X$ together with a monotone map $[\sigma]\colon X^{n}\to X$, for each specified $n$-ary operation $\sigma$. A \emph{homomorphism} of such algebras is a monotone map which preserves the operations. Then a \emph{variety} in this context is defined as a class of ordered algebras satisfying a set of formal inequalities $s\leq t$, where $s,t$ are $\Sigma$-terms. A \emph{quasivariety} is a class defined by more general formal implications of the form $\bigwedge\limits_{i\in I}(s_{i}\leq t_{i})\implies s\leq t$, where again the $s_{i},t_{i},s,t$ are $\Sigma$-terms.

The categories $\nc{OrdSGrp}$ and $\nc{OrdMon}$ of ordered semigroups and ordered monoids respectively are both examples of varieties which play an important role in the theory of automata. More generally, any quasivariety of ordinary algebras gives rise to a quasivariety of ordered algebras defined by the same axioms. A different example of quasivariety is given by the \emph{cancellative} ordered monoids $\nc{OrdMon_{can}}$, i.e. the ordered monoids $(M,\cdot,\leq)$ satisfying the implications $x\cdot z\leq y\cdot z\implies x\leq y$ and $z\cdot x\leq z\cdot y\implies x\leq y$ for all $x,y,z\in M$. A further source of examples is furnished by ordinary varieties whose axioms contain those of semi-lattices, since they can be equipped with the equationally definable order $x\leq y\iff x\vee y=y$. Yet more examples of quasivarieties are given by the \emph{Kleene algebras} of Logic.

While (quasi-)varieties of ordered algebras are a central source of examples of $\nc{Pos}$-enriched categories, there are other interesting examples that will appear in the present paper. For one, we have the category $S$-$\nc{Pos}$ \cite{S-Pos} of monotone actions of an ordered monoid $S$ on a poset and monotone equivariant maps between them. Furthermore, there are categories of ordered topological spaces, such as Priestley spaces or the \emph{compact ordered spaces} of Nachbin. These are all examples of categories which are either themselves monadic over $\nc{Pos}$ or reflective in a monadic category.

The thread of this paper can be seen as one continuation of the ideas developed in \cite{Kurz-Velebil} and is in part suggested by the authors at the end of the latter paper. At the same time it is part of a growing recent interest in the categorical treatment of ordered algebras, as for example in the recent preprints \cite{ADV} and \cite{AFMS}. Our main contribution here is a construction of the \emph{exact completion of a regular category} for $\nc{Pos}$-categories which employs a suitably enriched version of the \emph{calculus of relations}. We then identify varieties of ordered algebras which occur as such completions of corresponding (quasi-)varieties of ordered or unordered algebras. Furthermore, we prove that the exact completion of the category of \emph{Priestley spaces} is precisely the category of \emph{Nachbin spaces}. This provides an ordered version of the folklore result which identifies the category of compact Hausdorff spaces as the exact completion (in the ordinary sense) of the regular category of Stone spaces. In fact, it will follow by the same token that the exact completion of Stone spaces in the enriched sense is also the category of Nachbin spaces.

\subsection*{Organization of the paper}

In section 2 we collect some preliminaries involving regularity for categories enriched over $\nc{Pos}$, mostly for the convenience of the reader. There is only one original contribution here, \ref{simplifying regularity}, which provides a simplification of the definition of regularity that was presented in \cite{Kurz-Velebil}. More precisely, we prove that one of the defining conditions is a consequence of the other three and  can thus be omitted.

Section 3 discusses the main aspects of the calculus of relations which is available in any regular category. The main result in this section is \ref{Pos-enriched maps are morphisms}, which identifies the morphisms of a regular category as the left adjoints in a suitable bicategory of relations.

Section 4 represents the crux of the paper and is where we construct the exact completion of a regular category. After the initial definition, we prove in a sequence of steps that our proposed construction indeed satisfies the required properties, culminating in \ref{Universal property of ex/reg}. The arguments here make extensive use of the calculus of relations relying on the previous section.

In section 5 we characterize the embedding of a regular category into its completion. This is subsequently used to obtain examples of categories which arise as exact completions of one or more of their regular subcategories. In particular, we show that the category of Nachbin spaces is exact and can be obtained as the completion of either the category of Priestley or Stone spaces.

Finally, in section 6 we examine the relationship between the process of exact completion and that of taking internal posets in an ordinary category. We prove that, in a suitable sense, these two commute.

\section{Preliminaries on Regularity}

In this section we collect some preliminaries concerning the notion of regularity for categories enriched over the cartesian closed category $\nc{Pos}$ of posets and order-preserving functions, as defined by Kurz and Velebil in \cite{Kurz-Velebil}. After recalling some basic facts about finite limits, we reexamine the definition of regularity and observe that one of the conditions therein is in fact redundant.

Throughout the paper by `a category $\mathcal{C}$' we shall always mean a category that is enriched over the cartesian closed category $\nc{Pos}$ of partially ordered sets and monotone functions. Explicitly, this means that $\mathcal{C}$ is a category such that each $\Hom_{\mathcal{C}}(X,Y)$ is equipped with a partial order relation and such that composition of morphisms is order-preserving in each variable. If we wish to refer to categories in the usual non-enriched sense we will always use the adjective `ordinary'.

A functor $F\colon\ca{C}\to\ca{D}$ will always mean a $\nc{Pos}$-functor, i.e. an ordinary functor that furthermore preserves the order of morphisms.

Similarly, whenever we speak of limits or colimits in a category $\ca{C}$, these will always mean \emph{weighted} (co-)limits (also called \emph{indexed} (co-)limits in \cite{Kelly Enriched Categories}). We know from \cite{Kelly Enriched Categories} that completeness of a category $\mathcal{C}$, i.e. the existence of all small weighted limits, is equivalent to the existence in $\ca{C}$ of all small conical limits and all powers. The former of these can in turn be constructed via products and equalizers, so that $\ca{C}$ is complete if and only if it possesses products, equalizers and powers. Recall here that the power of an object $X\in\ca{C}$ to a poset $P$ is an object $X^{P}\in\ca{C}$ for which there exists a natural isomorphism
\begin{displaymath}
\ca{C}(C,X^{P})\cong\mathrm{Hom}_{\nc{Pos}}(P,\ca{C}(C,X))
\end{displaymath}

When the base of enrichment is locally finitely presentable as a monoidal category\cite{Finite Limits Enriched}, as in our case with $\nc{Pos}$, there is also a useful notion of \emph{finite} weighted limit. In particular, we have that $\ca{C}$ is finitely complete if and only if it has finite products, equalizers and finite powers. By \emph{finite power} here we mean a power object $X^{P}$ where $P$ is a finitely presentable object in $\nc{Pos}$, i.e. a finite poset.
\vspace{3mm}

We begin by recalling some basic notions, most of which can also be found in \cite{Kurz-Velebil}. First, the notion of monomorphism that is more appropriate in the ordered context and which will form part of the factorization system leading to the notion of regularity for $\nc{Pos}$-categories.

\begin{definition}
	A morphism $m\colon X\to Y$ in a category $\ca{C}$ is called an \emph{$\nc{ff}$-morphism} (or \emph{representably fully faithful}, or an \emph{order-monomorphism}) if for every $Z\in\ca{C}$ the monotone map $\ca{C}(Z,m)\colon\ca{C}(Z,X)\to\ca{C}(Z,Y)$ in $\nc{Pos}$ also reflects the order.
\end{definition}

We shall use the terms ``$\nc{ff}$-morphism'' and ``order-monomorphism'' interchangeably throughout the paper. Furthermore, we will use the term `order-epimorphism' for the dual notion.

Explicitly, $m\colon X\to Y$ is an $\nc{ff}$-morphism when for every $f,g\colon Z\to X$ the implication $mf\leq mg\implies f\leq g$ holds. In $\nc{Pos}$, $m$ is an $\nc{ff}$-morphism precisely when it is an order-embedding, i.e. a map which preserves and reflects the order. Any such map is of course a monomorphism, but the converse is not true.

The shift from monomorphisms to $\nc{ff}$-morphisms is also essentially the difference between the notion of conical (weighted) limit in a category $\ca{C}$ and the ordinary limit of the same type in the underlying ordinary category $\ca{C}_{0}$. For example, consider two objects $X,Y\in\ca{C}$. Then a diagram $\xy\xymatrix{X & X\times Y\ar[l]_{\pi_{X}}\ar[r]^{\pi_{Y}} & Y}\endxy$ is a product diagram if the usual unique factorization property is satisfied, along with the following additional condition: given any two morphisms $u,v\colon Z\to X\times Y$, the pair of inequalities $\pi_{X}u\leq\pi_{X}v$ and $\pi_{Y}u\leq\pi_{Y}v$ together imply that $u\leq v$. In other words, the pair of projections $\pi_{X},\pi_{Y}$ must be jointly \emph{order}-monomorphic, rather than just jointly monomorphic. This stems from the fact that the universal property is a natural isomorphism $\Hom(Z,X\times Y)\cong\Hom(Z,X)\times\Hom(Z,Y)$ in $\nc{Pos}$, rather than in $\nc{Set}$. A similar observation applies to colimits in $\mathcal{C}$. 

Let us record below a few basic properties of $\nc{ff}$-morphisms familiar for monomorphisms in an ordinary category.

\begin{lemma}
	Consider morphisms $f\colon X\to Y$ and $g\colon Y\to Z$ in a category $\ca{C}$. Then:
	\begin{enumerate}
		\item If $f,g$ are $\nc{ff}$-morphisms, then so is $gf$.
		\item If $gf$ is an $\nc{ff}$-morphism, then so is $f$.
		\item $\nc{ff}$-morphisms are stable under pullback.
	\end{enumerate}
\end{lemma}
\begin{proof}
	Perhaps only item (3) needs some details, so consider the following pullback square in $\ca{C}$ where $f$ is an $\nc{ff}$-morphism and assume that $u,v\colon A\to P$ are such that $qu\leq qv$.
	\begin{center}
		\hfil
		\xy\xymatrix{P\ar[r]^{p}\ar[d]_{q} & X\ar[d]^{f} \\
			Y\ar[r]_{g} & Z}\endxy
		\hfil 
	\end{center}
	We then have $gqu\leq gqv\implies fpu\leq fpv\implies pu\leq pv$, since $f$ is an $\nc{ff}$-morphism. Then, because we have both $pu\leq pv$ and $qu\leq qv$, we conclude by the limit property of the pullback that $u\leq v$.
\end{proof}

We recall next two particular examples of weighted limits which are not conical and which play an important role in the context of $\nc{Pos}$-categories. See also \cite{Kurz-Velebil}.

The \emph{comma object} of an ordered pair of morphisms $(f,g)$ with common codomain is a square
\begin{displaymath}
	\begin{tikzcd}[sep=0.35in]		C\ar[r,"c_{1}"]\ar[d,"c_{0}"']\ar[dr,phantom,"\leq"] & Y\ar[d,"g"] \\
	X\ar[r,"f"'] & Z
	\end{tikzcd}	
\end{displaymath}
such that $fc_{0}\leq gc_{1}$ and which is universal with this property, the latter meaning precisely the following two properties:
\begin{enumerate}
	\item Given $u_{0}\colon W\to X$ and $u_{1}\colon W\to Y$ in $\ca{C}$ such that $fu_{0}\leq gu_{1}$, there exists a $u\colon W\to C\in\ca{C}$ such that $c_{0}u=u_{0}$ and $c_{1}u=u_{1}$.
	\item The pair $(c_{0},c_{1})$ is jointly order-monomorphic.
\end{enumerate}
Note in particular that the factorization given in (1) will be unique by (2). We will usually denote the comma object $C$ by $f/g$. In $\nc{Pos}$, the comma object is given by $f/g=\{(x,y)\in X\times Y| f(x)\leq g(y)\}$ with the order induced from the product.

The \emph{inserter} of an ordered pair $(f,g)$ of parallel morphisms $\begin{tikzcd}X\ar[r,shift left=1ex,"f"]\ar[r,shift right=1ex, "g"']& Y\end{tikzcd}$ is a morphism $e\colon E\to X\in\ca{C}$ such that $fe\leq ge$ and universal in the following sense:
\begin{enumerate}
	\item If $h\colon Z\to X\in\ca{C}$ is such that $fh\leq gh$, then there exists a $u\colon Z\to A$ such that $eu=h$.
	\item $e$ is an $\nc{ff}$-morphism.
\end{enumerate}
Again, note that the factorization posited in (1) is unique by property (2). In $\nc{Pos}$, the inserter is precisely $E=\{x\in X| f(x)\leq g(x)\}$ with the order induced from $X$.

It will be convenient for us to have an alternative way of constructing all finite weighted limits using inserters along with some conical limits. This is then the content of the following proposition, which should be well-known and in any case follows from more general facts about 2-categorical limits. We nevertheless include a proof for the sake of completeness.

\begin{proposition}\label{finite limits=finite products+inserters}
	A category $\ca{C}$ is finitely complete if and only if it has finite products and inserters.
\end{proposition}
\begin{proof}
	Suppose that $\ca{C}$ has finite products and inserters. To have all finite conical limits it suffices to construct equalizers. So consider any parallel pair of morphisms $\xy\xymatrix{X\ar@<1ex>[r]^{f}\ar@<-1ex>[r]_{g} & Y}\endxy$ in $\ca{C}$. Let $m\colon M\to X$ be the inserter of the pair $(f,g)$ and then let $n\colon E\to M$ be the inserter of $(gm,fm)$. We claim that now $e\coloneqq mn\colon E\to X$ is the desired equalizer. Indeed, note first that $gmn\leq fmn$ and also $fm\leq gm\implies fmn\leq gmn$, so that $fmn=gmn$. Then suppose that $h\colon Z\to X$ is such that $fh=gh$. Since in particular $fh\leq gh$, there exists a unique $u\colon Z\to M$ such that $mu=h$. Now $u$ is such that $gmu=gh\leq fh=fmu$, so we have a unique $v\colon Z\to E$ such that $nv=u$. So $ev=mnv=mu=h$. Finally, it is clear that $e$ is order-monomorphic, since both $m,n$ are so.
	
	Second, we need to construct finite powers, so let $P$ be a finite poset and consider any $X\in\ca{C}$. Consider the product $\prod\limits_{a\in P}X$ (i.e. the ordinary power) and for every pair of elements $a,b\in P$ with $a\leq b$ form the inserter of $(\pi_{a},\pi_{b})$, say $\xy\xymatrix{E_{ab}\ar@{>->}[r]^{e_{ab}} & \prod\limits_{a\in P}X\ar@<1ex>[r]^{\pi_{a}}\ar@<-1ex>[r]_{\pi_{b}} & X}\endxy$ in $\ca{C}$. Set $E\coloneqq\prod\limits_{a\leq b}E_{ab}$. We claim the $E$ is the power $X^{P}$.
	
	To show this, consider any family of morphisms $(f_{a}\colon C\to X)_{a\in P}$ such that $a\leq b\implies f_{a}\leq f_{b}$, i.e. a homomorphism $P\to\ca{C}(C,X)$ in $\nc{Pos}$. There is then a unique $f\colon C\to\prod\limits_{a\in P}X$ such that $\pi_{a}f=f_{a}$ for all $a\in P$. Now for each pair of elements $a,b\in P$ with $a\leq b$ we have $f_{a}\leq f_{b}$, which is to say $\pi_{a}f\leq\pi_{b}f$. Hence, there is a unique $u_{ab}\colon C\to E_{ab}$ with $e_{ab}u_{ab}=f$. This in turn induces a unique $u\colon C\to E$ such that $\pi_{ab}u=u_{ab}$ whenever $a,b\in P$ with $a\leq b$.
	
	Finally, let's show that this assignment is order-preserving and order reflecting. So consider another family $(g_{a}\colon C\to X)_{a\in P}$, with $g\colon C\to\prod\limits_{a\in P}X$ and $v\colon C\to E$ corresponding to $f$ and $u$ as defined above for $(f_{a})_{a\in P}$. If $(f_{a})_{a\in P}\leq(g_{a})_{a\in P}$, then $f_{a}\leq g_{a}$ for all $a\in P$, i.e. $\pi_{a}f\leq \pi_{a}g$ for all $a\in P$ and hence $f\leq g$ by the universal property of the product. This in turn means that whenever $a\leq b$ we have $e_{ab}u_{ab}=f\leq g=e_{ab}v_{ab}$ and so $u_{ab}\leq v_{ab}$, hence $u\leq v$. It is clear that these implications can also be reversed.
\end{proof}
\vspace{3mm}

If a category $\ca{C}$ has comma objects, then in particular for any morphism $f\colon X\to Y\in\ca{C}$ we can form the comma of the pair $(f,f)$. This comma measures the extent to which $f$ fails to be an $\nc{ff}$-morphism and ultimately connects with the notions of regularity and exactness to be introduced shortly.

\begin{definition}
	Given any morphism $f\colon X\to Y$ in a category $\ca{C}$, the comma object $f/f$ is called the \emph{kernel congruence} of $f$.
\end{definition}

\begin{lemma}
	For any morphism $f\colon X\to Y$ with kernel congruence $\xy\xymatrix{f/f\ar@<1ex>[r]^{f_{0}}\ar@<-1ex>[r]_{f_{1}} & X}\endxy$ in a category $\mathcal{C}$, the following are equivalent:
	\begin{enumerate}
		\item $f$ is an $\nc{ff}$-morphism.
		\item $f_{0}\leq f_{1}$.
		\item The canonical morphism $\iota_{f}\colon 1_{X}/1_{X}\to f/f$ is an isomorphism.
	\end{enumerate}
\end{lemma}
\begin{proof}
	If $f$ is an $\nc{ff}$-morphism, then $ff_{0}\leq ff_{1}\implies f_{0}\leq f_{1}$.
	
	If $f_{0}\leq f_{1}$, then $1_{X}f_{0}\leq 1_{X}f_{1}$ implies that $(f_{0},f_{1})$ must factor through $1_{X}/1_{X}$ via a morphism which is then easily seen to be inverse to $1_{X}/1_{X}\rightarrow f/f$.
	
	Finally, assume that $f/f\cong 1_{X}/1_{X}$ and let $u_{0},u_{1}\colon Z\to X$ be such that $fu_{0}\leq fu_{1}$. Then $(u_{0},u_{1})$ factors through $f/f$ and so through $1_{X}/1_{X}$. But the latter means precisely that $u_{0}\leq u_{1}$.
\end{proof}

As we have mentioned already, the class of $\nc{ff}$-morphisms will be the ``mono part'' of a factorization system for regular categories. The other class of morphisms is taken to be the class of morphisms orthogonal to all $\nc{ff}$-morphisms, as has to be the case in any orthogonal factorization system. Let us first recall the definition of orthogonality in this enriched context.

\begin{definition}
	Given morphisms $e\colon A\to B$ and $m\colon X\to Y$ in a category $\ca{C}$, we say that $e$ is \emph{left orthogonal} to $m$ and write $e\perp m$ if the square
	\begin{center}
		\hfil
		\xy\xymatrix{\ca{C}(B,X)\ar[r]^{-\circ e}\ar[d]_{m\circ -} & \ca{C}(A,X)\ar[d]^{m\circ -} \\
			\ca{C}(B,Y)\ar[r]_{-\circ e} & \ca{C}(A,Y)}\endxy
		\hfil 
	\end{center}
	is a pullback in $\nc{Pos}$.
\end{definition}

To make things more explicit, the statement $e\perp m$ means two things:
\begin{enumerate}
	\item The usual diagonal fill-in property.
	\item Given two commutative squares
	$	\xy\xymatrix@=3em{A\ar[r]^{e}\ar[d]_{u_{1}} & B\ar[d]^{v_{1}}\ar@{-->}[dl]_{d_{1}} \\
		X\ar[r]_{m} & Y}\endxy
	\xy\xymatrix@=3em{A\ar[r]^{e}\ar[d]_{u_{2}} & B\ar[d]^{v_{2}}\ar@{-->}[dl]_{d_{2}} \\
		X\ar[r]_{m} & Y}\endxy$
	
	in $\ca{C}$ with $u_{1}\leq u_{2}$ and $v_{1}\leq v_{2}$, the diagonal fill-ins must also satisfy $d_{1}\leq d_{2}$.
\end{enumerate}	

Then, as in \cite{Kurz-Velebil} we introduce the following class of morphisms, which in some sense are the $\nc{Pos}$-enriched analogue of strong epimorphisms for ordinary categories.

\begin{definition}
	A morphism $e\colon A\to B$ is called an \emph{$\nc{so}$-morphism} (or \emph{surjective on objects}) if $e\perp m$ for every $\nc{ff}$-morphism $m\colon X\to Y\in\ca{C}$.
\end{definition}

\begin{remark}
	We note here that, in the special case of checking a condition $e\perp m$ with $m$ an $\nc{ff}$-morphism, the 2-dimensional part of the definition of orthogonality (property 2 above) follows for free. Indeed, in the notation of the above diagrams, the fact that $md_{1}=v_{1}\leq v_{2}=md_{2}$ then implies $d_{1}\leq d_{2}$.
\end{remark}
\vspace{3mm}

We record here two basic properties of $\nc{so}$-morphisms which will be used throughout the paper. These follow simply because the class of $\nc{so}$-morphisms is defined by a left orthogonality condition.

\begin{lemma}\label{properties of so-morphisms}
	Consider morphisms $f\colon X\to Y$ and $g\colon Y\to Z$ in a category $\ca{C}$. Then:
	\begin{enumerate}
		\item If $f,g$ are $\nc{so}$-morphisms, then so is $gf$.
		\item If $gf$ is an $\nc{so}$-morphism, then so is $g$.
	\end{enumerate}
\end{lemma}
\vspace{3mm}

The $\nc{so}$-morphisms will be the ``epimorphism part'' of the factorization system on regular categories. Indeed, given the existence of some limits, every $\nc{so}$-morphism is an order-epimorphism.

\begin{lemma}
	If $\ca{C}$ has inserters, then every $\nc{so}$-morphism in $\ca{C}$ is an \\ order-epimorphism.
\end{lemma}
\begin{proof}
	Let $e\colon A\to B$ be an $\nc{so}$-morphism and consider $f,g\colon B\to C$ such that $fe\leq ge$. Let $m\colon M\to B$ be the inserter of $(f,g)$. Then there exists a unique $u\colon A\to M$ such that $mu=e$. But we have $e\perp m$ and so we obtain a $v\colon B\to M$ such that $ve=u$, $mv=1_{B}$. Now $m$ is both a split epimorphism and a monomorphism, hence an isomorphism and so $fm\leq gm\implies f\leq g$.
\end{proof}

The notion of regularity for ordinary categories, as is well known, can be defined either in terms of strong epimorphisms or in terms of regular epimorphisms. In fact, one can argue that a significant part of the power of regularity is that it forces these two classes of epimorphisms to coincide. From another perspective, the regular epimorphisms are the categorical notion of quotient that is usually appropriate in ordinary categories. For $\nc{Pos}$-categories the corresponding notion is that of \emph{coinserter}, which is dual to the notion of inserter defined earlier. Intuitively, whereas taking a coequalizer corresponds to adding new equalities, constructing a coinserter should be thought of as adding new \emph{inequalities}.

Now, just as every regular epimorphism is always strong, so too we have the following.

\begin{lemma}
	Every coinserter is an $\nc{so}$-morphism.
\end{lemma}
\begin{proof}
	Consider the following commutative diagram, where $e$ is assumed to be the coinserter of $(f_{0},f_{1})$ and $m$ is an $\nc{ff}$-morphism.
	\begin{center}
		\hfil
		\xy\xymatrix{C\ar@<1ex>[r]^{f_{0}}\ar@<-1ex>[r]_{f_{1}} & A\ar[r]^{e}\ar[d]_{u} & B\ar[d]^{v} \\
			& X\ar[r]_{m} & Y}\endxy
		\hfil 
	\end{center}
	Now we have $muf_{0}=vef_{0}\leq vef_{1}=muf_{1}$, so by virtue of $m$ being an $\nc{ff}$-morphism we get $uf_{0}\leq uf_{1}$. Then by the coinserter property there exists a unique $d\colon B\to X$ such that $de=u$. It follows also that $md=v$ because $e$ is an order-epimorphism.
\end{proof}

\begin{definition}
	A morphism $q\colon X\to Y$ is called an \emph{effective} (epi-)morphism if it is the coinserter of some pair of parallel morphisms.
\end{definition}

The notion of regularity of a category is essentially an exactness property relating kernel congruences and coinserters. There are however some exactness properties involving these notions that always hold, regardless of regularity. This is the content of the following proposition, which again should be compared to the corresponding facts involving kernel pairs and coequalizers in an ordinary category (see \cite{Handbook}).

\begin{proposition}\label{Kernel Congruences and Coinserters}
	\begin{enumerate}
		\item If an effective epimorphism has a kernel congruence, then it must be the coinserter of that kernel congruence.
		\item If a kernel congruence has a coinserter, then it is also the kernel congruence of that coinserter.
	\end{enumerate}
\end{proposition}
\begin{proof}
	\begin{enumerate}
		\item Suppose $\xy\xymatrix{X\ar@<1ex>[r]^{f_{0}}\ar@<-1ex>[r]_{f_{1}} & Y\ar[r]^{q} & Q}\endxy$ is a coinserter diagram and assume $q$ has a kernel congruence  $\xy\xymatrix{q/q\ar@<1ex>[r]^{q_{0}}\ar@<-1ex>[r]_{q_{1}} & Y}\endxy$ in $\ca{C}$. Then $qf_{0}\leq qf_{1}\implies (\exists! u\colon X\to q/q)q_{0}u=f_{0}, q_{1}u=f_{1}$, by the universal property of kernel congruence.
		
		Now let $g\colon Y\to Z$ be such that $gq_{0}\leq gq_{1}$. Then $gq_{0}u\leq gq_{1}u\implies gf_{0}\leq gf_{1}$ and so $(\exists! v\colon Q\to Z)vq=g$. Finally, $q$ is already an order-epimorphism by virtue of being a coinserter of some pair of morphisms.
		
		\item Suppose that $\xy\xymatrix{R\ar@<1ex>[r]^{r_{0}}\ar@<-1ex>[r]_{r_{1}} & X\ar[r]^{q} & Q}\endxy$ is a coinserter diagram and that $(r_{0},r_{1})$ is the kernel congruence of some $f\colon X\to Y$. Then $fr_{0}\leq fr_{1}$ implies the existence of a unique $u\colon Q\to Y$ such that $uq=f$.
		
		Now let $a,b\colon A\to X$ be such that $qa\leq qb$. Then also $uqa\leq uqb$, i.e. $fa\leq fb$ and so $(\exists!v\colon A\to R)r_{0}v=a, r_{1}v=b$.
	\end{enumerate}
\end{proof}

We now come to the definition of regularity for $\nc{Pos}$-enriched categories, as presented by Kurz and Velebil in \cite{Kurz-Velebil}. We label it as `provisional' in the context of this paper for reasons that will be justified shortly.

\begin{definition}(provisional)\label{Kurz-Velebil regularity}
	A category $\ca{C}$ is \emph{regular} if it satisfies the following:\newline
	(R1) $\ca{C}$ has all finite (weighted) limits.\newline
	(R2) $\ca{C}$ has ($\nc{so}$,$\nc{ff}$)-factorizations.\newline
	(R3) $\nc{so}$-morphisms are stable under pullback in $\ca{C}$.\newline
	(R4) Every $\nc{so}$-morphism is effective in $\ca{C}$.
\end{definition}

A main feature of this definition is that it posits the existence of a stable ($\nc{so}$,$\nc{ff}$) factorization system in $\ca{C}$. The authors go on to state that the `gist of the definition' is property (R4), i.e. the assumption that $\nc{so}$-morphisms and effective morphisms coincide, which essentially states that it is equivalent to require the stable factorization system to be (effective,$\nc{ff}$). However, we shall show that condition (R4) in fact follows from the first three conditions, much like in the case of ordinary regularity one can state the definition equivalently either in terms of regular epimorphisms or strong epimorphisms. In fact, the proof is essentially a direct adaptation of the corresponding one in the ordinary context (see \cite{Handbook}).

Before making good on our claim, we need a preparatory result on the pasting of a pullback square with a comma square. This is well-known from the realm of 2-categories, but we include its proof for the sake of making this paper more self-contained.


\begin{lemma}\label{pullback-comma lemma}
	Consider the following diagram in a category $\ca{C}$, where the right-hand square is a comma square and the left-hand square commutes.
	\begin{center}
		\hfil
		\xy\xymatrix{P\ar[r]^{p_{1}}\ar[d]_{p_{0}} & Q\ar[r]^{q_{1}}\ar[d]_{q_{0}}\ar@{}[dr]|{\leq} & Z\ar[d]^{g} \\
			X'\ar[r]_{x} & X\ar[r]_{f} & Y}\endxy
		\hfil 
	\end{center}
	Then the outer rectangle is a comma square if and only if the left-hand square is a pullback.
\end{lemma}
\begin{proof}
	Assume first that the left-hand square is a pullback. Note first that by our assumptions on the diagram we have $fxp_{0}=fq_{0}p_{1}\leq gq_{1}p_{1}$. Now suppose that $u\colon A\to X'$ and $v\colon A\to Z$ are such that $fxu\leq gv$. Since the right-hand square is a comma, $(\exists! w\colon A\to Q)q_{0}w=xu, q_{1}w=v$. The first of these equalities by virtue of the pullback property gives that $(\exists! z\colon A\to P)p_{0}z=u, p_{1}z=w$. Then we have $q_{1}p_{1}z=q_{1}w=v$ as well.
	
	Finally, assume that $z,z'\colon A\to P$ are such that $p_{0}z\leq p_{0}z'$ and $q_{1}p_{1}z\leq q_{1}p_{1}z'$. Then we also have $xp_{0}z\leq xp_{0}z'\implies q_{0}p_{1}z\leq q_{0}p_{1}z'$, so that by the universal property of the comma square we get $p_{1}z\leq p_{1}z'$. The latter inequality together with $p_{0}z\leq p_{0}z'$ yield $z\leq z'$ by the universal property of the pullback.
	\vspace{3mm}
	
	Conversely, assume that the outer square is a comma and let $u\colon A\to X'$ and $v\colon A\to Q$ be such that $xu=q_{0}v$. Then we have $fxu=fq_{0}v\leq gq_{1}v$, so the outer rectangle being a comma says that $(\exists! w\colon A\to P)p_{0}w=u, q_{1}p_{1}w=q_{1}v$. But since also $q_{0}p_{1}w=xp_{0}w=xu=q_{0}v$ and $q_{0},q_{1}$ are jointly monomorphic, we obtain also that $p_{1}w=v$. Finally, it is clear that $p_{0},p_{1}$ are jointly order-monomorphic because $p_{0},q_{1}p_{1}$ are so.
\end{proof}

Now we can prove that condition (R4) in the definition of regularity is superfluous. The proof that follows is almost identical to that of Proposition 2.2.2 in \cite{Handbook}, concerning ordinary regularity, where we replace some uses of the familiar lemma on pasting of pullback squares with \ref{pullback-comma lemma}.

\begin{proposition}\label{simplifying regularity}
	If $\ca{C}$ is a category satisfying conditions (R1),(R2) and (R3) of \ref{Kurz-Velebil regularity}, then it also satisfies (R4).
\end{proposition}
\begin{proof}
	Let $f\colon X\to Y$ be an $\nc{so}$-morphism and consider its kernel congruence $\xy\xymatrix{f/f\ar@<1ex>[r]^{k_{0}}\ar@<-1ex>[r]_{k_{1}} & X}\endxy$, which exists in $\ca{C}$ by (R1). We want to show that $f$ is the coinserter of $(k_{0},k_{1})$.
	
	So let $g\colon X\to Z\in\ca{C}$ be such that $gk_{0}\leq gk_{1}$. We can consider then the induced morphism $\langle f,g\rangle\colon X\to Y\times Z$. By (R2) we can factor this morphism as an $\nc{so}$-morphism followed by an $\nc{ff}$-morphism, say $\xy\xymatrix{X\ar@{->>}[r]^{p} & I\ar@{>->}[r]^{i} & Y\times Z}\endxy$. We then form the following diagram in $\ca{C}$, where we begin by forming the bottom right-hand square as a comma square and then the remaining three squares are pullbacks.
	\begin{center}
		\hfil
		\xy\xymatrix@=3em{P\ar[r]^{v_{1}}\ar[d]_{v_{0}} & P_{1}\ar[r]^{x_{1}}\ar[d]_{u_{1}} & X\ar@{->>}[d]^{p} \\
			P_{0}\ar[r]^{u_{0}}\ar[d]_{x_{0}} & C\ar[r]^{c_{1}}\ar[d]_{c_{0}}\ar@{}[dr]|{\leq} & I\ar[d]^{\pi_{Y}i}\\
			X\ar@{->>}[r]_{p} & I\ar[r]_{\pi_{Y}i} & Y}\endxy
		\hfil 
	\end{center}
	By an application of \ref{pullback-comma lemma} and its order-dual as well as the usual pullback gluing lemma we deduce that the big outer square resulting from the pasting of all four smaller ones is also a comma square. Then, since $\pi_{Y}ip=f$, we have $P\cong f/f$ and we can assume that $x_{0}v_{0}=k_{0}$ and $x_{1}v_{1}=k_{1}$. Also, observe that by (R3) we have that $u_{0},u_{1}$ and then also $v_{0},v_{1}$ are $\nc{so}$-morphisms.
	
	Now we want to show that $\pi_{Y}i$ is an iso. Since $(\pi_{Y}i)p=f$ is an $\nc{so}$-morphism, we already know, by \ref{properties of so-morphisms}, that $\pi_{Y}i$ is an $\nc{so}$-morphism as well. Thus, it suffices to show that it is also an $\nc{ff}$-morphism, which is equivalent to showing $c_{0}\leq c_{1}$. Since $u_{0}v_{0}=u_{1}v_{1}$ is an $\nc{so}$-morphism, so in particular an order-epimorphism, the latter inequality is equivalent to having $c_{0}u_{0}v_{0}\leq c_{1}u_{0}v_{0}$. This is in turn equivalent to $ic_{0}u_{0}v_{0}\leq ic_{1}u_{0}v_{0}$, because $i$ is an $\nc{ff}$-morphism. To prove this last inequality we now observe the following:
	\begin{displaymath}
	\pi_{Y}ic_{0}u_{0}v_{0}=\pi_{Y}ipx_{0}v_{0}=fx_{0}v_{0}=fk_{0}\leq fk_{1}=fx_{1}v_{1}=\pi_{Y}ipx_{1}v_{1}=\pi_{Y}ic_{1}u_{0}v_{0}
	\end{displaymath}
	\begin{displaymath}
	\pi_{Z}ic_{0}u_{0}v_{0}=\pi_{Z}ipx_{0}v_{0}=gx_{0}v_{0}=gk_{0}\leq gk_{1}=gx_{1}v_{1}=\pi_{Z}ipx_{1}v_{1}=\pi_{Z}ic_{1}u_{0}v_{0}
	\end{displaymath}
	Then the universal property of the product yields the desired inequality.
	
	Finally, we now have a morphism $\pi_{Z}i(\pi_{Y}i)^{-1}\colon Y\to Z$ such that $\pi_{Z}i(\pi_{Y}i)^{-1}f=\pi_{Z}i(\pi_{Y}i)^{-1}\pi_{Y}ip=\pi_{Z}ip=g$. Furthermore, we know already that $f$ is an order-epimorphism because it is an $\nc{so}$-morphism by assumption.
\end{proof}

Thus, we can officially strike condition (R4) from the definition of regularity and henceforth adopt the following more economical one.

\begin{definition}\label{our regularity}
	A category $\ca{C}$ will be called \emph{regular} if it satisfies the following:\newline
	(R1) $\ca{C}$ has all finite (weighted) limits.\newline
	(R2) $\ca{C}$ has ($\nc{so}$,$\nc{ff}$)-factorizations.\newline
	(R3) $\nc{so}$-morphisms are stable under pullback in $\ca{C}$.
\end{definition}



Similarly, we can now furthermore establish another equivalent characterization of regularity in terms of the existence of quotients for kernel congruences.

\begin{proposition}
	A finitely complete category $\ca{C}$ is regular if and only if the following hold:
	\begin{enumerate}
		\item Every kernel congruence in $\ca{C}$ has a coinserter.
		\item Effective morphisms are stable under pullback in $\ca{C}$.
	\end{enumerate}
\end{proposition}
\begin{proof}
	If $\ca{C}$ is regular, then it is easy to see that it satisfies the two conditions above by definition and by an appeal to part (2.) of \ref{Kernel Congruences and Coinserters}.
	
	Conversely, let us assume that $\ca{C}$ satisfies the two conditions in the statement. Consider any $f\colon X\to Y\in\ca{C}$, its kernel congruence $\xy\xymatrix{f/f\ar@<1ex>[r]^{f_{0}}\ar@<-1ex>[r]_{f_{1}} & X}\endxy$ and the coinserter $q\colon X\to Q$ in $\ca{C}$ of the latter, which exists by condition 1. Since $ff_{0}\leq ff_{1}$, there exists a unique $m\colon Q\to Y$ such that $f=mq$. It now suffices to show that $m$ is an $\nc{ff}$-morphism.
	
	For this, consider the kernel congruence $\xy\xymatrix{m/m\ar@<1ex>[r]^{m_{0}}\ar@<-1ex>[r]_{m_{1}} & Q}\endxy$ and form the following diagram where the bottom right-hand square is a comma and the other three are pullbacks.
	\begin{center}
		\hfil
		\xy\xymatrix@=3em{P\ar[r]^{v_{1}}\ar[d]_{v_{0}} & P_{1}\ar[r]^{x_{1}}\ar[d]_{u_{1}} & X\ar@{->>}[d]^{q} \\
			P_{0}\ar[r]^{u_{0}}\ar[d]_{x_{0}} & m/m\ar[r]^{m_{1}}\ar[d]_{m_{0}}\ar@{}[dr]|{\leq} & Q\ar[d]^{m}\\
			X\ar@{->>}[r]_{q} & Q\ar[r]_{m} & Y}\endxy
		\hfil 
	\end{center}
	Similarly to the proof of \ref{simplifying regularity}, we have $P\cong f/f$ and we can assume that $x_{0}v_{0}=f_{0}$ and $x_{1}v_{1}=f_{1}$. Furthermore, by the assumed stability of effective morphisms under pullback, we can deduce that $u_{0}v_{0}=u_{1}v_{1}$ is an order-epimorphism. Now we have that
	\begin{displaymath}
	m_{0}u_{0}v_{0}=qx_{0}v_{0}=qf_{0}\leq qf_{1}=qx_{1}v_{1}=m_{1}u_{1}v_{1}
	\end{displaymath}
	whence we deduce that $m_{0}\leq m_{1}$ and so that $m$ is an $\nc{ff}$-morphism.
\end{proof}

To end this section, let us list a few examples of regular categories. For more details on most of these one can consult \cite{Kurz-Velebil}.

\begin{example}\label{examples of regular}
	\begin{enumerate}
		\item $\nc{Pos}$ is regular as a $\nc{Pos}$-category. So is any category of enriched presheaves $[\ca{C}^{op},\nc{Pos}]$ for $\ca{C}$ a small category.
		\item Any ordinary regular category $\ca{C}$ is also regular in the $\nc{Pos}$-enriched sense when equipped with the discrete order on its Hom-sets. Indeed, in this case $\nc{ff}$-morphisms coincide with monomorphisms and $\nc{so}$-morphisms with strong epimorphisms. Note that $\nc{Pos}$ is an example of a category which is not regular in the ordinary sense, but is regular as an enriched category.
		\item \emph{Quasivarieties of ordered algebras} in the sense of Bloom and Wright\cite{Bloom & Wright} are regular categories\cite{Kurz-Velebil}. As particular examples here we have the categories $\nc{OrdMon}$ of ordered monoids, $\nc{OrdSGrp}$ of ordered semi-groups, $\nc{OrdCMon}$ of commutative ordered monoids and $\nc{OrdMon_{0}}$ of ordered monoids with the neutral element $0$ of the monoid operation as the minimum element for the order. These are all in fact varieties. An example of a quasivariety which is not a variety is the category $\nc{OrdMon_{can}}$ of cancellative monoids, i.e. those ordered monoids $(M,\cdot,\leq)$ satisfying the implications $x\cdot z\leq y\cdot z\implies x\leq y$ and $z\cdot x\leq z\cdot y\implies x\leq y$ for all $x,y,z\in M$.
		\item The categories $\nc{Nach}$ of \emph{Nachbin} spaces (or compact ordered spaces) and $\nc{Pries}$ of \emph{Priestley} spaces with continuous order-preserving functions in both instances are examples of regular categories. We shall have more to say on these in section 5.
		\item If $\ca{C}$ is monadic over $\nc{Pos}$ for a monad $T\colon\nc{Pos}\to\nc{Pos}$ which preserves $\nc{so}$-morphisms (i.e. surjections), then $\ca{C}$ is regular. An example of this kind is given by the category $S$-$\nc{Pos}$ of $S$-posets for any ordered monoid $S$ (see \cite{S-Pos}). The objects in the latter category are monoid actions $S\times X\to X$ on a poset $X$ which are monotone in both variables, while the morphisms are the monotone equivariant functions. 
	\end{enumerate}
\end{example}

\section{Calculus of Relations and Exactness}

Recall that in an ordinary regular category $\ca{C}$, the existence of a stable (regular epi,mono) factorization system allows for a well-behaved calculus of relations in $\ca{C}$. More precisely, the existence of the factorization system allows one to define the composition of two internal relations and then the stability of regular epimorphisms under pullback is precisely equivalent to the associativity of this composition. Essentially the same facts hold also in our $\nc{Pos}$-enriched setting.

If $\ca{E}$ is a regular category, then by a \emph{relation} in $\ca{E}$ we shall mean an order-subobject $R\rightarrowtail X\times Y$, i.e. a subobject of a product represented by an $\nc{ff}$-morphism. We shall write $R\colon X\looparrowright Y$ to denote that $R$ is a relation from $X$ to $Y$ in $\ca{E}$. The factorization system ($\nc{so}$,$\nc{ff}$) in $\ca{E}$ and the stability of $\nc{so}$-morphisms under pullback allow us to have a well-defined composition of relations in $\ca{E}$: given $R\colon X\looparrowright Y$ and $S\colon Y\looparrowright Z$ in $\ca{C}$, the relation $S\circ R\colon X\looparrowright Z$ is defined by first constructing the pullback square below
\begin{center}
	\hfil
	\xy\xymatrix{ & & T\ar[dl]_{t_{0}}\ar[dr]^{t_{1}} & & \\
		& R\ar[dl]_{r_{0}}\ar[dr]^{r_{1}} & & S\ar[dl]_{s_{0}}\ar[dr]^{s_{1}} & \\
		X & & Y & & Z}\endxy
	\hfil 
\end{center}
and then taking the ($\nc{so}$,$\nc{ff}$) factorization of $\langle r_{0}t_{0},s_{1}t_{1}\rangle\colon T\to X\times Z$, say
\begin{center}
\hfil
\xy\xymatrix@=4em{T\ar@{->>}[r]^{q} & M\ar@{>->}[r]^{\langle m_{0},m_{1}\rangle} & X\times Z}\endxy
\hfil
\end{center}
In other words, in our notation $S\circ R$ is the relation represented by the $\nc{ff}$-morphism $\langle m_{0},m_{1}\rangle$.

This leads to the locally posetal bicategory (a.k.a $\nc{Pos}$-category) $\mathrm{Rel}(\ca{E})$, whose objects are those of $\ca{E}$ and whose morphisms are the relations $R\colon X\looparrowright Y$ in $\ca{E}$. The identity morphism on the object $X$ in $\mathrm{Rel}(\ca{E})$ is the diagonal relation $\langle 1_{X},1_{X}\rangle\colon\Delta_{X}\rightarrowtail X\times X$ and composition of morphisms is given by composition of relations in $\ca{E}$.

If we forget about the 2-dimensional nature of the properties that define the two classes of morphisms in the factorization system ($\nc{so}$,$\nc{ff}$), then the structure and basic properties of $\mathrm{Rel}(\ca{E})$ are essentially the calculus of relations \emph{relative to a stable factorization system}, as explicated by Meisen \cite{Meisen}, Richter\cite{Richter}, Kelly\cite{Kelly Relations relative to factorizations} and others. Thus, we shall feel free to take for granted many of the basic facts concerning the structure of $\mathrm{Rel}(\ca{E})$ without proving them here. As an exception to this rule, we include the proof of the following lemma because it describes a way in which one can argue about relations in a regular category using generalized elements which will be particularly useful to us in subsequent proofs. Recall here that, given a relation $R\colon X\looparrowright Y$ and generalized elements $x\colon A\to X$, $y\colon A\to Y$ in $\ca{E}$, we write $(x,y)\in_{A}R$ to indicate that $\langle x,y\rangle\colon A\to X\times Y$ factors through $\langle r_{0},r_{1}\rangle$. Observe also that, given an $\nc{so}$-morphism $q\colon B\twoheadrightarrow A$ in $\ca{E}$, we have that $(x,y)\in_{A}R$ if and only if $(xq,yq)\in_{B}R$. Indeed, while the ``only if'' direction is obvious, for the converse note that $(xq,yq)\in_{B}R$ means the existence of a commutative square of the following form.
\begin{center}
\begin{tikzcd}
	B\ar[r,two heads,"q"]\ar[d] & A\ar[d,"{\langle x,y\rangle}"] \\
	R\ar[r,tail] & X\times Y
\end{tikzcd}
\end{center}
Then by the orthogonality between $\nc{so}$ and $\nc{ff}$-morphisms we have an induced diagonal $A\to R$ exhibiting $(x,y)\in_{A}R$.

\begin{lemma}
	Let $R\colon X\looparrowright Y$ and $S\colon Y\looparrowright Z$ be relations in a regular category $\ca{E}$ and consider any generalized elements $x\colon P\to X$ and $z\colon P\to Z$. Then $(x,z)\in_{P}S\circ R$ if and only if there exists an effective epimorphism $q\colon Q\twoheadrightarrow P$ and a generalized element $y\colon Q\to Y$ such that $(xq,y)\in_{Q}R$ and $(y,zq)\in_{Q}S$.
\end{lemma}
\begin{proof}
	Consider the diagram below, where the square is a pullback.
	\begin{center}
		\hfil
		\xy\xymatrix{ & & T\ar[dl]_{t_{0}}\ar[dr]^{t_{1}} & & \\
			& R\ar[dl]_{r_{0}}\ar[dr]^{r_{1}} & & S\ar[dl]_{s_{0}}\ar[dr]^{s_{1}} & \\
			X & & Y & & Z}\endxy
		\hfil 
	\end{center}
	Then $S\circ R$ is given by the following image factorization
	$$\langle r_{0}t_{0},s_{1}t_{1}\rangle=\xy\xymatrix@=3.5em{T\ar@{->>}[r]^{e} & I\ar@{>->}[r]^>>>>>>>{\langle i_{0},i_{1}\rangle} & X\times Z}\endxy$$.
	
	Assume first that $(x,z)\in_{P}S\circ R$, i.e. there exists a morphism $u\colon P\to I$ such that $\langle i_{0},i_{1}\rangle u=\langle x,z\rangle$. We then form the pullback square below.
	\begin{center}
		\hfil
		\xy\xymatrix{Q\ar[d]_{v}\ar@{->>}[r]^{q} & P\ar[d]^{u} \\
			T\ar@{->>}[r]_{e} & I}\endxy
		\hfil 
	\end{center}
	Note that $q$ is an effective epimorphism because $e$ is such. Now set $y\coloneqq r_{1}t_{0}v=s_{0}t_{1}v\colon Q\to Y$. Then $\langle xq,y\rangle=\langle i_{0}uq,r_{1}t_{0}v\rangle=\langle i_{0}ev,r_{1}t_{0}v\rangle=\langle r_{0}t_{0}v,r_{1}t_{0}v\rangle=\langle r_{0},r_{1}\rangle t_{0}v$ and $\langle y,zq\rangle=\langle s_{0}t_{1}v,i_{1}uq\rangle=\langle s_{0}t_{1}v,i_{1}ev\rangle=\langle s_{0}t_{1}v,s_{1}t_{1}v\rangle=\langle s_{0},s_{1}\rangle t_{1}v$, so that $(xq,y)\in_{Q}R$ and $(y,zq)\in_{Q}S$.
	
	Conversely, assume that $(xq,y)\in_{Q}R$ and $(y,zq)\in_{Q}S$ for some $y\colon Q\to Y$ and effective epimorphism $q\colon Q\twoheadrightarrow P$. This means that there exist morphisms $u\colon Q\to R$ and $v\colon Q\to S$ such that $\langle r_{0},r_{1}\rangle u=\langle xq,y\rangle$ and $\langle s_{0},s_{1}\rangle v=\langle y,zq\rangle$. Since $r_{1}u=y=s_{0}v$, there exists a unique $w\colon Q\to T$ such that $t_{0}w=u$ and $t_{1}w=v$. Then $\langle i_{0},i_{1}\rangle ew=\langle r_{0}t_{0},s_{1}t_{1}\rangle w=\langle r_{0}u,s_{1}v\rangle=\langle xq,zq\rangle$, so that $(xq,zq)\in_{Q}SR$. Since $q$ is an $\nc{so}$-morphism, we can conclude that also $(x,z)\in_{P}SR$.
\end{proof}

The above lemma can actually be used to prove many of the fundamental properties of $\rm{Rel}(\ca{E})$. In general, $\rm{Rel}(\ca{E})$ is a \emph{tabular allegory with a unit} (see e.g. \cite{Elephant},\cite{Freyd-Scedrov}), where the anti-involution $(-)^{\circ}\colon\mathrm{Rel}(\ca{E})^{op}\to\mathrm{Rel}(\ca{E})$ is given by taking the \emph{opposite} relation. In particular, we have that Freyd's \emph{Modular Law} holds in $\mathrm{Rel}(\ca{E})$, i.e. 
\begin{displaymath}\label{Modular Law}
	QP\cap S\subseteq Q(P\cap Q^{\circ}S) \tag{ML}
\end{displaymath}
for any relations $P\colon X\looparrowright Y$, $Q\colon Y\looparrowright Z$ and $S\colon X\looparrowright Z$ in $\ca{E}$. The presence of the anti-involution $(-)^{\circ}$ implies that the Modular Law is equivalent to its dual form, namely the inclusion
\begin{displaymath}\label{Modular Law*}
QP\cap S\subseteq (Q\cap SP^{\circ})P \tag{ML*}
\end{displaymath}
for any relations $P,Q,S$ as above.

Every morphism $f\colon X\to Y\in\ca{E}$ defines a relation $X\looparrowright Y$ represented by the ff-morphism $\langle 1_{X},f\rangle\colon X\to X\times Y$, which we call its \emph{graph} and denote by the same letter. This assignment defines a faithful ordinary functor $\ca{E}_{0}\to\mathrm{Rel}(\ca{E})$ on the underlying ordinary category of $\ca{E}$ which is the identity on objects. Furthermore, in $\mathrm{Rel}(\ca{E})$ we have an adjunction $f\dashv f^{\circ}$, which means that the inclusions $f^{\circ}f\supseteq\Delta_{X}$ and $ff^{\circ}\subseteq\Delta_{Y}$ hold. We say then that the morphisms of $\ca{E}$ are \emph{maps} in the bicategory $\mathrm{Rel}(\ca{E})$. We should perhaps stress here that taking the graph of a morphism does not define a functor $\ca{E}\to\mathrm{Rel}(\ca{E})$ because the order of morphisms is not preserved. In fact, since $\mathrm{Rel}(\ca{E})$ is an allegory, the anti-involution $(-)^{\circ}$ forces any inclusion $f\subseteq g$ for morphisms $f,g\colon X\to Y$ to be an equality (see e.g. A3.2.3 in \cite{Elephant}).

The modular law also implies some restricted versions of distributivity of composition over binary intersections in $\mathrm{Rel}(\ca{E})$ (see e.g. A3.1.6 of \cite{Elephant}). Two particular instances of this which we would like to explicitly record for future reference are the following:
\begin{displaymath}\label{Map distributivity}
(R\cap S)f= Rf\cap Sf \tag{MD}
\end{displaymath}
\begin{displaymath}\label{Map distributivity*}
g^{\circ}(R\cap S)=g^{\circ}R\cap g^{\circ}S \tag{MD*}
\end{displaymath}
where $R,S$ are relations $Y\looparrowright Z$ and $f\colon X\to Y$ and $g\colon X\to Z$ are morphisms in $\ca{E}$.

\vspace{3mm}

While $\mathrm{Rel}(\ca{E})$ is a very useful category in terms of performing calculations with relations in $\ca{E}$, it does not really capture the enriched nature of $\ca{E}$. As we have mentioned earlier, $\mathrm{Rel}(\ca{E})$ in some sense only involves an ordinary category with a stable factorization system. In order to also capture the $\nc{Pos}$-enriched aspects of a regular category $\ca{E}$ we will also need to work in a different bicategory of relations. In terms of our goals in this paper, this need is related to our desire to identify the morphisms of $\ca{E}$ as the maps in a certain bicategory of relations, thus generalizing a familiar fact for ordinary regular categories. Indeed, while certainly any morphism $f\colon X\to Y\in\ca{E}$ has as its right adjoint in $\mathrm{Rel}(\ca{E})$ the opposite relation $f^{\circ}$, this is not a complete characterization of (graphs of) morphisms. As can be deduced by arguments essentially contained in \cite{Kelly Relations relative to factorizations}, being a map in $\mathrm{Rel}(\ca{E})$ is a weaker property than being the graph of a morphism in $\ca{E}$. 

Another reason for moving to a different bicategory of relations is dictated by the form of \emph{congruences} in this enriched setting and the way in which exactness is defined. This will become apparent a little bit later.

\begin{definition}
	A relation $R\colon X\looparrowright Y$ in the regular category $\ca{E}$ is called \emph{weakening} or \emph{weakening-closed} if, whenever $x,x'\colon A\to X$ and $y,y'\colon A\to Y$ are generalized elements in $\ca{E}$, the following  implication holds 
	\begin{displaymath}
	x'\leq x\hspace{1mm}\wedge\hspace{1mm} (x,y)\in_{A}R\hspace{1mm}\wedge\hspace{1mm} y\leq y'\implies (x',y')\in_{A}R
	\end{displaymath}
\end{definition}
\vspace{3mm}

\begin{remark}
	The property of a relation being weakening-closed can be viewed as an order-compatibility condition, especially if one thinks of a relation $R\colon X\to Y$ in this poset-enriched context as specifying that certain elements of $X$ are less than or equal to certain elements of $Y$. If one follows the terminology of 2-category theory, as the authors of \cite{Kurz-Velebil} do, this property would be referred to by saying that $R$ is a \emph{two-sided discrete fibration}. Since we have generally chosen to not really stress the 2-categorical viewpoint in this paper, we have accordingly adopted the above terminology which is inspired from logic.
\end{remark}

\vspace{3mm}

For any given object $X\in\ca{E}$, there is a weakening-closed relation $I_{X}\colon X\looparrowright X$ given by the comma $I_{X}\coloneqq 1_{X}/1_{X}$. Then it is easy to see that a relation $R\colon X\looparrowright Y$ is weakening-closed if and only if we have $R=I_{Y}RI_{X}$ in $\mathrm{Rel}(\ca{E})$. In particular, the relations $I_{X}$ act as identity elements for composition of weakening-closed relations. Thus, we can define a bicategory $\mathrm{Rel}_{w}(\ca{E})$ where the objects are again those of $\ca{E}$ but now the morphisms are the weakening-closed relations. 

\begin{remark}
	Perhaps we should note here that, although every morphism of $\mathrm{Rel}_{w}(\ca{E})$ is also a morphism in $\mathrm{Rel}(\ca{E})$ and composition in both categories is the same, this is not a functorial inclusion $\mathrm{Rel}_{w}(\ca{E})\hookrightarrow\mathrm{Rel}(\ca{E})$ because identity morphisms are not preserved.
\end{remark}

Now to any given morphism $f\colon X\to Y\in\ca{E}$ we can canonically associate two weakening-closed relations $f_{*}\colon X\looparrowright Y$ and $f^{*}\colon Y\looparrowright X$ via the following commas: $f_{*}\coloneqq f/1_{Y}$ and $f^{*}\coloneqq 1_{Y}/f$. We sometimes call $f_{*}$ and $f^{*}$ the \emph{hypergraph} and \emph{hypograph} of $f$ respectively. In terms of generalized elements $x\colon A\to X$ and $y\colon A\to Y$ in $\ca{E}$ we have $(x,y)\in_{A}f_{*}\iff fx\leq y$ and $(y,x)\in_{A}f^{*}\iff y\leq fx$.

The following are then easy to see, for example by arguing with generalized elements.

\begin{lemma}
	Let $\ca{E}$ be a regular category. Then for any $f\colon X\to Y$ and $g\colon Y\to Z$ in $\ca{E}$ we have
	\begin{enumerate}
		\item $(gf)_{*}=g_{*}f_{*}$.
		\item $(gf)^{*}=f^{*}g^{*}$.
	\end{enumerate}
\end{lemma}

It is also easy to see that, for any $f,g\colon X\to Y\in\ca{E}$, we have equivalences
\begin{displaymath}
f\leq g\iff g_{*}\subseteq f_{*}\iff f^{*}\subseteq g^{*} 
\end{displaymath}
We can thus define two fully order-faithful functors $(-)_{*}\colon\ca{E}^{co}\hookrightarrow\mathrm{Rel}_{w}(\ca{E})$ and $(-)^{*}\colon\ca{E}^{op}\hookrightarrow\mathrm{Rel}_{w}(\ca{E})$, where $\ca{E}^{co}$ denotes the \lq\lq order-dual\rq\rq\space category, i.e. the category obtained from $\ca{E}$ by reversing the order on morphisms. Similarly, arguing with generalized elements we easily deduce the following.

\begin{lemma}
	For any $f\colon X\to Y$ in a regular category $\ca{E}$ we have $f^{*}f_{*}=f/f$ as relations in $\ca{E}$.
\end{lemma}

In particular, we see that $f\colon X\to Y$ is an $\nc{ff}$-morphism in $\ca{E}$ if and only if $f^{*}f_{*}=I_{X}$ in $\mathrm{Rel}_{w}(\ca{E})$. Similarly, a pair of morphisms $\begin{tikzcd}Y & X\ar[l,"f"']\ar[r,"g"] & Z\end{tikzcd}$ is jointly $\nc{ff}$ precisely when $f^{*}f_{*}\cap g^{*}g_{*}=I_{X}$.

Now just as the graph of every morphism $f\colon X\to Y$ induces an adjunction $f\dashv f^{\circ}$ in $\rm{Rel}(\mathcal{E})$, so do the hypergraph and hypograph of that morphism form an adjunction in the bicategory $\rm{Rel}_{w}(\mathcal{E})$.

\begin{lemma}
	For any $f\colon X\to Y$ in a regular category $\ca{E}$, there is an adjunction $f_{*}\dashv f^{*}$ in $\mathrm{Rel}_{w}(\ca{E})$.
\end{lemma}
\begin{proof}
	We saw above that $f^{*}f_{*}$ is precisely the kernel congruence of $f$, so clearly we have $I_{X}\subseteq f^{*}f_{*}$. To form the composition $f_{*}f^{*}$ we consider the diagram below, where the top square is a pullback and the other two are commas.
	\begin{center}
		\hfil 
		\xy\xymatrix{ & & Q\ar[dl]_{q_{1}}\ar[dr]^{q_{2}} & &  \\
			& f^{*}\ar[dl]_{s_{1}}\ar[dr]^{s_{2}} & & f_{*}\ar[dl]_{r_{1}}\ar[dr]^{r_{2}} & \\
			Y\ar@{=}[dr]\ar@{}[rr]|{\leq} & & X\ar[dl]_{f}\ar[dr]^{f}\ar@{}[rr]|{\leq} & & Y\ar@{=}[dl]  \\
			& Y & & Y & }\endxy
		\hfil 
	\end{center}
	Now by definition we have that $f_{*}f^{*}$ is given by the image of $\langle s_{1}q_{1},r_{2}q_{2}\rangle$. But $s_{1}q_{1}\leq fs_{2}q_{1}=fr_{1}q_{2}\leq r_{2}q_{2}$, so that $\langle s_{1}q_{1},r_{2}q_{2}\rangle$ factors through $I_{Y}=1_{Y}/1_{Y}$. This yields the inclusion $f_{*}f^{*}\subseteq I_{Y}$.
\end{proof}

In other words, every morphism $f\colon X\to Y$ in a regular category $\ca{E}$ is a map in $\mathrm{Rel}_{w}(\ca{E})$ via its hypergraph $f_{*}$. Our first result in this paper is that in any regular category $\ca{E}$ this is now indeed a complete characterization of morphisms, i.e. every map in $\rm{Rel}_{w}(\ca{E})$ is of the form $f_{*}$ for a (necessarily unique) morphism $f\colon X\to Y\in\ca{E}$.

\begin{theorem}\label{Pos-enriched maps are morphisms}
	If $\phi\colon X\looparrowright Y\in\rm{Rel}_{w}(\ca{E})$ has a right adjoint in $\rm{Rel}_{w}(\mathcal{E})$, then there exists a (necessarily unique) morphism $f\colon X\to Y\in\ca{E}$ such that $\phi=f_{*}$.
\end{theorem}
\begin{proof}
	Let $\psi\colon Y\looparrowright X\in\mathrm{Rel}_{w}(\ca{E})$ be the right adjoint, so that we have $I_{X}\subseteq \psi\phi$ and $\phi\psi\subseteq I_{Y}$. Suppose also that $\phi$ and $\psi$ are represented respectively by the ff-morphisms $\langle\phi_{0},\phi_{1}\rangle\colon T\rightarrowtail X\times Y$ and $\langle\psi_{0},\psi_{1}\rangle\colon T'\rightarrowtail Y\times X$. We next form the pullback square below, so that $\langle\phi_{0}u,\phi_{1}u\rangle=\langle\psi_{1}u',\psi_{0}u'\rangle\colon S\rightarrowtail X\times Y$ represents the relation $\phi\cap\psi^{\circ}\colon X\looparrowright Y\in\rm{Rel}(\ca{E})$. We first want to show that $\phi_{0}u=\psi_{1}u'$ is an isomorphism, in which case we will have $\phi\cap\psi^{\circ}=f$ for the morphism $f\coloneqq\phi_{1}u(\phi_{0}u)^{-1}\colon X\to Y$.
	\begin{center}
		\hfil
		\xy\xymatrix@=3em{S\ar@{>->}[r]^{u'}\ar@{>->}[d]_{u} & T'\ar@{>->}[d]^{\langle\psi_{1},\psi_{0}\rangle} \\
			T\ar@{>->}[r]_{\langle\phi_{0},\phi_{1}\rangle} & X\times Y}\endxy
		\hfil 
	\end{center}
	
	First of all, since $I_{X}\subseteq \psi\phi$, we have $(1_{X},1_{X})\in_{X}\psi\phi$ and hence there exist an effective epimorphism $e\colon P\twoheadrightarrow X$ and a $y\colon P\to Y$ such that $(e,y)\in_{P}\phi$ and $(y,e)\in_{P}\psi$. Then we have $(e,y)\in_{P}\phi\cap\psi^{\circ}$ and so there exists a $v\colon P\to S$ such that $\langle\phi_{0}u,\phi_{1}u\rangle v=\langle e,y\rangle$. In particular, since $\phi_{0}uv=e$ is an effective epimorphism, we deduce that $\phi_{0}u$ is an effective epimorphism as well.
	
	Now it suffices to show that $\phi_{0}u$ is also an $\nc{ff}$-morphism. So let $a,b\colon A\to S$ be such that $\phi_{0}ua\leq\phi_{0}ub$. Then $(\phi_{0}ua,\phi_{0}ub)\in_{A}I_{X}$ and hence $(\phi_{0}ua,\phi_{0}ub)\in_{A}\psi\phi$, so there is an effective epimorphism $e\colon P\twoheadrightarrow A$ and a $z\colon P\to Y$ such that $(\phi_{0}uae,z)\in_{P}\phi$ and $(z,\phi_{0}ube)\in_{P}\psi$. Since we also clearly have
	\begin{displaymath}
	(\phi_{1}uae,\phi_{0}uae)=(\psi_{0}u'ae,\psi_{1}u'ae)\in_{P}\psi,
	\end{displaymath}
	we get $(\phi_{1}uae,z)\in_{P}\phi\psi$, which implies $(\phi_{1}uae,z)\in_{P}I_{Y}$, i.e. that $\phi_{1}uae\leq z$.
	
	Similarly, since $(\phi_{0}ube,\phi_{1}ube)\in_{P}\phi$ and $(z,\phi_{0}ube)\in\psi$, we consequently have that $(z,\phi_{1}ube)\in_{P}\phi\psi$ and hence $(z,\phi_{1}ube)\in_{P}I_{Y}$, which is to say that $z\leq\phi_{1}ube$. We now have $\phi_{1}uae\leq z\leq\phi_{1}ube$, hence $\phi_{1}uae\leq\phi_{1}ube$, which in turn implies that $\phi_{1}ua\leq\phi_{1}ub$, because $e$ is an order-epimorphism. Since also  $\phi_{0}ua\leq\phi_{0}ub$, we obtain $a\leq b$ because $\langle\phi_{0}u,\phi_{1}u\rangle$ is an $\nc{ff}$-morphism.
	
	We now claim that $\phi=f_{*}$. To see this, observe first that $f_{*}=I_{Y}f=I_{Y}(\phi\cap\psi^{\circ})\subseteq I_{Y}\phi=\phi$, where the last equality follows because $\phi$ is weakening-closed. Thus, $f_{*}\subseteq\phi$. But by an application of the modular law \ref{Modular Law} in $\mathrm{Rel}(\ca{E})$ we also have $\psi f=\psi(\phi\cap\psi^{\circ})\supseteq\psi\phi\cap \Delta_{X}\supseteq I_{X}\cap\Delta_{X}=\Delta_{X}$ and so that $\phi\psi f\supseteq\phi$. Then $f_{*}=I_{Y}f\supseteq\phi\psi f\supseteq\phi$ and so we conclude that $f_{*}=\phi$.
	
	Finally, for the uniqueness claim let us assume that $f,g\colon X\to Y$ are such that $f_{*}=g_{*}$. By an earlier observation we know that the inclusion $f_{*}\subseteq g_{*}$ is equivalent to the inequality $g\leq f$. Then we have both $f\leq g$ and $g\leq f$, whence $f=g$.
\end{proof}
\vspace{5mm}

Next, let us comment here on how the calculus of relations in a regular category $\mathcal{E}$ can be used to express various limit properties therein. For example, the statement that a pair of morphisms $\begin{tikzcd}X & Z\ar[l,"f"']\ar[r,"g"] & Y\end{tikzcd}$ represents a given relation $R\colon X\looparrowright Y$ is equivalent to the following two equalities between relations:
\begin{enumerate}
	\item $R=gf^{\circ}$.
	\item $f^{*}f_{*}\cap g^{*}g_{*}=I_{Z}$.
\end{enumerate}
Note here that we cannot replace condition (1.) by $R=g_{*}f^{*}$, even if $R$ is weakening-closed. Similarly, we cannot replace (2.) by $f^{\circ}f\cap g^{\circ}g=\Delta_{Z}$, because the latter means that $(f,g)$ are only jointly monomorphic instead of jointly order-monomorphic.

Based on the above observation, now consider the statement that a commutative square
\begin{displaymath}
\begin{tikzcd}
	P\ar[r,"p_{1}"]\ar[d,"p_{0}"'] & Y\ar[d,"g"] \\
	X\ar[r,"f"'] & Z
\end{tikzcd}
\end{displaymath}
is a pullback. It is easy to see that this is equivalent to the following pair of equalities:
\begin{enumerate}
	\item $g^{\circ}f=p_{1}p_{0}^{\circ}$.
	\item $p_{1}^{*}p_{1*}\cap p_{0}^{*}p_{0*}=I_{P}$.
\end{enumerate}
\vspace{2mm}

Similarly, the statement that the square
\begin{displaymath}
	\begin{tikzcd}
		P\ar[r,"p_{1}"]\ar[d,"p_{0}"']\ar[dr,phantom,"\leq"] & Y\ar[d,"g"] \\
		X\ar[r,"f"'] & Z
	\end{tikzcd}
\end{displaymath}
is a comma is equivalent to:
\begin{enumerate}
	\item $g^{*}f_{*}=p_{1}p_{0}^{\circ}$.
	\item $p_{1}^{*}p_{1*}\cap p_{0}^{*}p_{0*}=I_{P}$.
\end{enumerate}
\vspace{5mm}

Now we turn to discussing exactness for $\nc{Pos}$-categories. First, let us recall the definition of congruence relation from \cite{Kurz-Velebil}. This can be seen as the ordered analogue of equivalence relations in an ordinary category. It is also a special case of a more general notion of congruence for 2-categories.

\begin{definition}
	Let $X$ be an object of the regular category $\ca{E}$. A \emph{congruence} on $X$ is a relation $E\colon X\looparrowright X\in\mathrm{Rel}_{w}(\ca{E})$ which is reflexive and transitive. 
	
	We say that the congruence $E$ is \emph{effective} if there exists a morphism $f\colon X\to Y\in\ca{E}$ such that $E=f/f$.
\end{definition}

Equivalently, we can say that $E\colon X\looparrowright X$ is a congruence if it is a transitive relation such that $E\supseteq I_{X}$. In essence, a congruence is a pre-order relation on $X$ which is compatible with the canonical order relation on $X$, the latter expressed by the requirement that it is weakening-closed. We think of a congruence as imposing additional inequalities on $X$, just as an equivalence relation corresponds to the idea of imposing new equalities. 

With the notion of congruence in hand, we are lead naturally to the notion of (Barr-)exactness for $\nc{Pos}$-categories as considered by Kurz-Velebil in \cite{Kurz-Velebil}.

\begin{definition}
	A regular category $\ca{E}$ is called \emph{exact} if every congruence in $\ca{E}$ is effective.
\end{definition}

\begin{example}
	\begin{enumerate}
		\item $\nc{Pos}$ is exact and so is any presheaf category $[\ca{C},\nc{Pos}]$ for any small category $\ca{C}$.
		\item The locally discrete category $\nc{Set}$ is an example of a category which is regular but not exact (see \cite{Kurz-Velebil}).
		\item Generalizing the case of $\nc{Pos}$, any \emph{variety of ordered algebras}, always in the sense of Bloom \& Wright\cite{Bloom & Wright}, is an exact category. Particular examples here are furnished by the categories $\nc{OrdSGrp}$, $\nc{OrdMon}$, $\nc{OrdMon_{0}}$ and $\nc{OrdCMon}$. On the other hand, the quasivariety $\nc{OrdCMon_{t.f.}}$ is not exact.
		\item There are also examples of exact categories which are not varieties, but are monadic over $\nc{Pos}$. One such, which will appear in more detail in section 5, is the category $\nc{Nach}$ of Nachbin spaces. Another is given by the category $S$-$\nc{Pos}$ for any ordered monoid $S$.
	\end{enumerate}
\end{example}

A congruence in a regular category $\ca{E}$, being a transitive relation, is an idempotent when considered as a morphism in either of $\mathrm{Rel}(\ca{E})$ and $\mathrm{Rel}_{w}(\ca{E})$. When it is moreover effective, then it actually is a \emph{split} idempotent in the latter bicategory. Indeed, if $E=f/f$ for some $f\colon X\to Y$, then we can assume that $f$ is actually the coinserter of $E$ by \ref{Kernel Congruences and Coinserters}. Then we have $f^{*}f_{*}=f/f=E$ and $f_{*}f^{*}=I_{Y}$. The next proposition shows that this splitting actually characterizes effective congruences. This is analogous to a familiar fact for ordinary regular categories, where an equivalence relation splits in the bicategory of relations if and only if it occurs as a kernel pair.

\begin{proposition}\label{effective=splitting}
	Let $E\colon X\looparrowright X$ be a congruence in the regular category $\ca{E}$. Then $E$ is effective if and only if it splits as an idempotent in $\mathrm{Rel}_{w}(\ca{E})$.
\end{proposition}
\begin{proof}
	Suppose that $E=\psi\phi$ for some $\phi\colon X\looparrowright Y$ and $\psi\colon Y\looparrowright X$ in $\mathrm{Rel}_{w}(\ca{E})$ with $\phi\psi=I_{Y}$. Since $E$ is reflexive and weakening-closed, we have $\psi\phi=E\supseteq I_{X}$. Since also trivially $\phi\psi\subseteq I_{Y}$, we have $\phi\dashv\psi$ in $\mathrm{Rel}_{w}(\ca{E})$ and so by \ref{Pos-enriched maps are morphisms} we have that $\phi=f_{*}$ and $\psi=f^{*}$ for some morphism $f\colon X\to Y\in\ca{E}$. Thus, we obtain $E=\psi\phi=f^{*}f_{*}=f/f$.
\end{proof}
\vspace{3mm}

In the following section we will embark on the goal of constructing the exact completion $\ca{E}_{ex/reg}$ of a regular category $\ca{E}$ by a process of splitting idempotents in a bicategory of relations and then taking maps in the resulting bicategory. This is essentially an attempt to mimic the construction of the ordinary exact completion of a regular category, as initially described by Lawvere in \cite{Perugia notes} and then with more details for example in \cite{SC}, \cite{Freyd-Scedrov},\cite{Elephant}, motivated by the combination of the two results contained in \ref{Pos-enriched maps are morphisms} and \ref{effective=splitting}.

However, even if the reader is not familiar with the construction of the splitting of idempotents, the next proposition can serve as a type of heuristic for coming up with the definition of morphisms in the completion $\ca{E}_{ex/reg}$. It will also be of practical use a little bit later. Recall here that in a regular category we say $\xy\xymatrix{E\ar@<1ex>[r]^{e_{0}}\ar@<-1ex>[r]_{e_{1}} & X\ar@{->>}[r]^{p} & P}\endxy$ is an \emph{exact sequence} if $(e_{0},e_{1})$ is the kernel congruence of $p$ and also $p$ is the coinserter of $(e_{0},e_{1})$. In terms of the calculus of relations, exactness of the sequence is equivalent to the equalities $p^{*}p_{*}=E$ and $p_{*}p^{*}=I_{P}$.

\begin{proposition}\label{morphisms between quotients of congruences}
	Let $\xy\xymatrix{E\ar@<1ex>[r]\ar@<-1ex>[r] & X\ar@{->>}[r]^{p} & P}\endxy$ and $\xy\xymatrix{F\ar@<1ex>[r]\ar@<-1ex>[r] & Y\ar@{->>}[r]^{q} & Q}\endxy$ be exact sequences in the regular category $\mathcal{E}$. Then there is an order-reversing bijection between the following:
	\begin{enumerate}
		\item Morphisms $P\to Q$ in $\ca{E}$.
		\item Relations $R_{*}\colon X\looparrowright Y\in\mathrm{Rel}_{w}(\ca{E})$ for which there exists another relation $R^{*}\colon Y\looparrowright X\in\mathrm{Rel}_{w}(\ca{E})$ such that the following are satisfied:
		\begin{itemize}
			\item $FR_{*}E=R_{*}$ and $ER^{*}F=R^{*}$.
			\item $R^{*}R_{*}\supseteq E$ and $R_{*}R^{*}\subseteq F$.
		\end{itemize}
	\end{enumerate}
\end{proposition}
\begin{proof}
	Consider first a morphism $r\colon P\to Q\in\ca{E}$. Set $R_{*}\coloneqq q^{*}r_{*}p_{*}$ and $R^{*}\coloneqq p^{*}r^{*}q_{*}$. We have $FR_{*}E=q^{*}q_{*}q^{*}r_{*}p_{*}p^{*}p_{*}=q^{*}r_{*}p_{*}=R_{*}$ and similarly $ER^{*}F=p^{*}p_{*}p^{*}r^{*}q_{*}q^{*}q_{*}=p^{*}r^{*}q_{*}=R^{*}$. In addition, $R^{*}R_{*}=p^{*}r^{*}q_{*}q^{*}r_{*}p_{*}=p^{*}r^{*}I_{Q}r_{*}p_{*}=p^{*}r^{*}r_{*}p_{*}\supseteq p^{*}p_{*}=E$ and $R_{*}R^{*}=q^{*}r_{*}p_{*}p^{*}r^{*}q_{*}\subseteq q^{*}r_{*}r^{*}q_{*}\subseteq q^{*}q_{*}=F$.
	
	Conversely, consider a relation $R_{*}$ as in (2.). Set $\phi\coloneqq q_{*}R_{*}p^{*}$ and then also $\psi\coloneqq p_{*}R^{*}q^{*}$. Then for these weakening-closed relations we have
	\begin{displaymath}
	\psi\phi=p_{*}R^{*}q^{*}q_{*}R_{*}p^{*}\supseteq p_{*}R^{*}R_{*}p^{*}\supseteq p_{*}Ep^{*}=I_{P}
	\end{displaymath} 
	\begin{displaymath}
	\phi\psi=q_{*}R_{*}p^{*}p_{*}R^{*}q^{*}=q_{*}R_{*}ER^{*}q^{*}=q_{*}R_{*}R^{*}q^{*}\subseteq q_{*}Fq^{*}=I_{Q}
	\end{displaymath}
	Thus, by \ref{Pos-enriched maps are morphisms} we have $\phi=r_{*}$ for a (unique) morphism $r\colon P\to Q$.
	
	The fact that these two assignments are inverse to each other is expressed by the two equalities $q_{*}q^{*}r_{*}p_{*}p^{*}=r_{*}$ and $q^{*}q_{*}R_{*}p^{*}p_{*}=FR_{*}E=R_{*}$.
\end{proof}
\vspace{5mm}

Before ending this section, let us make a couple more observations on the assignment $R_{*}\mapsto r$ from the last proposition. Note that, as the notation suggests and the above proof exhibits, the relation $R_{*}$ corresponds to the hypergraph $r_{*}$ of a morphism $r\colon P\to Q$. Specifically, $r$ is uniquely determined from $R_{*}$ by the equality $r_{*}=q_{*}R_{*}p^{*}$. We would like to also record here a relation $R$ that in some sense corresponds directly to the (graph of the) morphism $r$.

Given $R_{*}\colon X\looparrowright Y$ as in \ref{morphisms between quotients of congruences}, set $R\coloneqq R_{*}\cap(R^{*})^{\circ}\colon X\looparrowright Y$. Then we claim that $r=qRp^{\circ}$. To see this, observe first that $qRp^{\circ}$ is also a map in $\mathrm{Rel}(\ca{E})$ because 
\begin{displaymath}
(qRp^{\circ})^{\circ}qRp^{\circ}=pR^{\circ}q^{\circ}qRp^{\circ}\supseteq pR^{\circ}Rp^{\circ}\supseteq p(E\cap E^{\circ})p^{\circ}=\Delta_{P}
\end{displaymath}
\begin{displaymath}
(qRp^{\circ})(qRp^{\circ})^{\circ}=qRp^{\circ}pR^{\circ}q^{\circ}=qR(E\cap E^{\circ})R^{\circ}q^{\circ}=qRR^{\circ}q^{\circ}\subseteq q(F\cap F^{\circ})q^{\circ}=\Delta_{Q}
\end{displaymath}
Here we used the fact that in an exact sequence $\xy\xymatrix{E\ar@<1ex>[r]\ar@<-1ex>[r] & X\ar@{->>}[r]^{p} & P}\endxy$ we have $E\cap E^{\circ}=p^{*}p_{*}\cap(p^{*}p_{*})^{\circ}=p^{\circ}p$, i.e. $E\cap E^{\circ}$ is precisely the kernel pair of $p$. In addition, the inclusions $R^{\circ}R\supseteq E\cap E^{\circ}$ and $RR^{\circ}\subseteq F\cap F^{\circ}$ were used, the first of which is clear and the second follows by the modular law (see the next section for details).

Now, finally, we have 
\begin{displaymath}
r=r_{*}\cap (r^{*})^{\circ}=q_{*}R_{*}p^{*}\cap (q^{*})^{\circ}(R^{*})^{\circ}(p_{*})^{\circ}\supseteq q(R_{*}\cap (R^{*})^{\circ})p^{\circ}=qRp^{\circ}. 
\end{displaymath}
But since we have an inclusion between two maps in the allegory $\mathrm{Rel}(\ca{E})$, these maps must be equal. Hence, we conclude that $r=qRp^{\circ}$.

\vspace{5mm}

\section{Exact Completion}

In this section we come to the heart of this paper, which is the construction of the exact completion of a regular $\nc{Pos}$-category $\ca{E}$. 

The main idea is to try to perform a construction that mimics one of the ways in which one can define the exact completion of an ordinary regular category $\ca{C}$. Let us thus quickly recall this construction, as originally suggested by Lawvere in \cite{Perugia notes}. We will also very much be drawing inspiration from the presentation of Succi-Cruciani\cite{SC}.

Given a regular ordinary category $\ca{C}$, one first performs a \emph{splitting of idempotents} in the bicategory of relations $\mathrm{Rel}(\ca{C})$. More precisely, one splits the class of equivalence relations, which are indeed idempotent as morphisms in $\mathrm{Rel}(\ca{C})$. This step yields a bicategory which, a posteriori, is identified as the bicategory of relations $\mathrm{Rel}(\ca{C}_{ex/reg})$ of the completion. The second step is then to identify the completion itself, which can be done by taking the category of maps in the bicategory produced by the first step. 

We note that the idea for this construction of the exact completion can be traced back to the following two observations, valid in any ordinary regular category $\ca{C}$:
\begin{enumerate}
	\item An equivalence relation $E$ on an object $X\in\ca{C}$ is effective if and only if it splits as an idempotent in $\mathrm{Rel}(\ca{C})$.
	\item The morphisms $f\colon X\to Y\in\ca{C}$ are precisely the maps in $\mathrm{Rel}(\ca{C})$.
\end{enumerate}

Accordingly, our hope to perform a $\nc{Pos}$-enriched version of this construction hinges on the validity of enriched versions of the two observations above, as contained respectively in \ref{effective=splitting} and \ref{Pos-enriched maps are morphisms}. Hence, in our context, we should first look at $\mathrm{Rel}_{w}(\ca{E})$, split the idempotents therein which are congruences in $\ca{E}$, then finally take the category of maps in the resulting bicategory.

The fact that we need to work with $\mathrm{Rel}_{w}(\ca{E})$ rather than $\mathrm{Rel}(\ca{E})$ already presents some issues. As we have mentioned earlier, $\mathrm{Rel}(\ca{E})$ has the structure of an allegory and it is this fact that facilitates many computations. Furthermore, the theory of allegories is well developed and in fact there is a precise correspondence between ordinary regular and exact categories on the one hand and certain classes of allegories on the other (see e.g. \cite{Elephant},\cite{Freyd-Scedrov}). On the contrary, the structure of $\mathrm{Rel}_{w}(\ca{E})$ is not as rich. Fundamentally, the process of taking the opposite of a relation does not restrict to $\mathrm{Rel}_{w}(\ca{E})$. Thus, in our quest to construct the $\nc{Pos}$-enriched exact completion as indicated above, we cannot simply rely on the general theory of allegories. While this creates a complication, at the same time it is in some sense to be expected. Indeed, allegories are in some aspects too simple for our enriched context. For example, the only inclusions between maps in an allegory are equalities and it is precisely this fact that does not allow us to recover the order relation on morphisms from $\mathrm{Rel}(\ca{E})$.

\vspace{5mm}

Motivated by the above, we embark towards our goal by first defining a category $Q_{w}(\ca{E})$ by splitting the idempotents in $\mathrm{Rel}_{w}(\ca{E})$ which are congruences in $\ca{E}$. Explicitly, $Q_{w}(\ca{E})$ is defined as follows:
\begin{itemize}
	\item \underline{Objects} of $Q_{w}(\ca{E})$ are pairs $(X,E)$, where $X$ is an object of $\ca{E}$ and $E\colon X\looparrowright X$ is a congruence relation in $\ca{E}$.
	\item \underline{Morphisms} $\Phi\colon(X,E)\to (Y,F)$ in $Q_{w}(\ca{E})$ are (weakening-closed) relations $\Phi\colon X\looparrowright Y$ in $\ca{E}$ such that $\Phi E=\Phi=F\Phi$ or equivalently $\Phi=F\Phi E$.
\end{itemize}
Composition in $Q_{w}(\ca{E})$ is composition of relations in $\ca{E}$, while the identity morphism on $(X,E)\in Q_{w}(\ca{E})$ is the relation $E$ itself. The morphisms are locally ordered by inclusion and it is clear that $Q_{w}(\ca{E})$ also has binary infima of morphisms given by intersection of relations in $\ca{E}$. Then we define a category $\ca{E}_{ex/reg}$ by taking the maps in $Q_{w}(\ca{E})$.

\begin{definition}
	$\ca{E}_{ex/reg}\coloneqq \mathrm{Map}(Q_{w}(\ca{E}))$.
\end{definition}

Explicitly, $\ca{E}_{ex/reg}$ has the same objects as $Q_{w}(\ca{E})$, while its morphisms are those $R_{*}\colon(X,E)\to (Y,F)\in Q_{w}(\ca{E})$ for which there exists an $R^{*}\colon(Y,F)\to (X,E)\in Q_{w}(\ca{E})$ such that $R^{*}R_{*}\supseteq E$ and $R_{*}R^{*}\subseteq F$.

Note that $\ca{E}_{ex/reg}$ is not merely an ordinary category, but can be made into a legitimate Pos-category by defining for any $R_{*},S_{*}\colon(X,E)\to (Y,F)$ in $\ca{E}_{ex/reg}$
\begin{displaymath}
R_{*}\leq S_{*}\vcentcolon\Leftrightarrow R_{*}\supseteq S_{*}
\end{displaymath} 
the inclusion on the right-hand side being that of relations in $\ca{E}$. This is clearly a partial order relation on Homs that is preserved by composition. Observe furthermore that we can equivalently define the order $R_{*}\leq S_{*}$ by requiring $R^{*}\subseteq S^{*}$ for the right adjoints.

We also have a canonical functor $\Gamma\colon\ca{E}\to\ca{E}_{ex/reg}$ defined by mapping an object $X\in\ca{E}$ to $(X,I_{X})$ and a morphism $f\colon X\to Y\in\ca{E}$ to its hypergraph $f_{*}\colon X\looparrowright Y$ considered as a morphism $(X,I_{X})\to(Y,I_{Y})$. Note that $\Gamma$ is order-preserving and reflecting by definition of the order on morphisms in $\ca{E}_{ex/reg}$.

\begin{remark}
	We will consistently denote morphisms of $\ca{E}_{ex/reg}$ by a capital letter with a lower asterisk and their right adjoint in $Q_{w}(\ca{E})$ by the same letter with an upper asterisk. This notation represents our intuition that $R_{*}$ is the hyper-graph of the morphism $R$ and in some sense we are working towards making this a precise statement. In particular, we will denote the identity morphism $(X,E)\to(X,E)$ by $1_{(X,E)*}$, where as relations in $\mathcal{E}$ we have $1_{(X,E)*}=E$ and $1_{(X,E)}^{*}=E$.
\end{remark}

Our goal now in this section is to show that the category $\ca{E}_{ex/reg}$ as defined above is the \emph{exact completion} of $\ca{E}$ as a regular category. The category $Q_{w}(\ca{E})$ will be seen in the end to be precisely the category of weakening-closed relations in $\ca{E}_{ex/reg}$. Accordingly, the proofs of the various statements about $\ca{E}_{ex/reg}$ later on in this section will be motivated by the description of the various limit and exactness properties in terms of the calculus of relations in a regular category. However, we know that such a description cannot be achieved with only weakening-closed relations. Thus, it will be convenient to construct also at the same time what will turn out to be the bicategory of all relations in $\ca{E}_{ex/reg}$. 

This leads us to define another bicategory $Q(\ca{E})$ as follows:
\begin{itemize}
	\item \underline{Objects} of $Q(\ca{E})$ are again pairs $(X,E)$, where $E\colon X\looparrowright X$ is a congruence relation in $\ca{E}$.
	\item \underline{Morphisms} $(X,E)\to (Y,F)$ in $Q(\ca{E})$ are relations $\Phi\colon X\looparrowright Y$ in $\ca{E}$ such that $\Phi(E\cap E^{\circ})=\Phi=(F\cap F^{\circ})\Phi$ or equivalently $(F\cap F^{\circ})\Phi(E\cap E^{\circ})=\Phi$.
\end{itemize}
The composition of morphisms is that of relations in $\ca{E}$ while the identity on $(X,E)\in Q(\ca{E})$ is $E\cap E^{\circ}$.

In other words, $Q(\ca{E})$ is the locally ordered bicategory obtained from $\mathrm{Rel}(\ca{E})$ by splitting those idempotents of the form $E\cap E^{\circ}$ for a congruence $E$ in $\ca{E}$. Notice that idempotents of this form are equivalence relations in $\ca{E}$ and in particular are symmetric. It then follows (see for example Theorem 3.3.4 in \cite{Elephant}) that $Q(\mathcal{E})$ is an allegory, where the opposite of a morphism is given by taking the opposite relation in $\ca{E}$.
\vspace{5mm}

Now we make some important observations regarding the connection between morphisms of $Q_{w}(\ca{E})$ and $Q(\ca{E})$ and between maps in these two bicategories. If the reader keeps in mind the intuition that these two categories should respectively be $\mathrm{Rel}_{w}(\ca{E}_{ex/reg})$ and $\mathrm{Rel}(\ca{E}_{ex/reg})$, then these observations are to be expected.

First, note that every morphism $\Phi\colon(X,E)\to (Y,F)$ in $Q_{w}(\ca{E})$ can also be considered as a morphism in $Q(\ca{E})$, since
\begin{displaymath}
\Phi=\Delta_{Y}\Phi\Delta_{X}\subseteq (F\cap F^{\circ})\Phi(E\cap E^{\circ})\subseteq F\Phi E=\Phi
\end{displaymath} 
However, it is important to note as well that this assignment is not functorial as it does not preserve the identity morphisms.

Second, to any map in $Q_{w}(\ca{E})$, i.e. to any morphism of $\ca{E}_{ex/reg}$, we can associate in a natural way a map in $Q(\ca{E})$ as follows.
 
\begin{lemma}
Consider any $R_{*}\colon(X,E)\to (Y,F)\in\ca{E}_{ex/reg}$ and define the relation $\mathfrak{gr}(R_{*})\coloneqq R_{*}\cap(R^{*})^{\circ}\colon X\looparrowright Y$ in $\ca{E}$. Then $\mathfrak{gr}(R_{*})$ is a map $(X,E)\to (Y,F)$ in $Q(\ca{E})$.
\end{lemma}
\begin{proof}
First, it is easy to see that $(F\cap F^{\circ})\mathfrak{gr}(R_{*})(E\cap E^{\circ})=\mathfrak{gr}(R_{*})$. Indeed, we have
\begin{displaymath}
\mathfrak{gr}(R_{*})=\Delta_{Y}\mathfrak{gr}(R_{*})\Delta_{X}\subseteq (F\cap F^{\circ})\mathfrak{gr}(R_{*})(E\cap E^{\circ})\subseteq (F\mathfrak{gr}(R_{*})E)\cap (F^{\circ}\mathfrak{gr}(R_{*})E^{\circ}) 
\end{displaymath}
\begin{displaymath}
\subseteq FR_{*}E\cap F^{\circ}(R^{*})^{\circ}E^{\circ}=R_{*}\cap(R^{*})^{\circ}=\mathfrak{gr}(R_{*})
\end{displaymath}

Second, to see that $\mathfrak{gr}(R_{*})$ is indeed a map in $Q(\ca{E})$ we argue as follows:
\begin{displaymath}
\mathfrak{gr}(R_{*})\mathfrak{gr}(R_{*})^{\circ}=(R_{*}\cap(R^{*})^{\circ})((R_{*})^{\circ}\cap R^{*})\subseteq R_{*}R^{*}\cap (R^{*})^{\circ}(R_{*})^{\circ}\subseteq F\cap F^{\circ}
\end{displaymath}
\begin{eqnarray*}
	\mathfrak{gr}(R_{*})^{\circ}\mathfrak{gr}(R_{*}) & = & (R_{*}^{\circ}\cap R^{*})(R_{*}\cap(R^{*})^{\circ})=(R_{*}^{\circ}\cap R^{*})((R_{*}\cap(R^{*})^{\circ})\cap R_{*}) \\
	& = & (R_{*}^{\circ}\cap R^{*})((R_{*}^{\circ}\cap R^{*})^{\circ}(E\cap E^{\circ})\cap R_{*}) \\
	& \stackrel{\ref{Modular Law}}{\supseteq} & (E\cap E^{\circ})\cap (R_{*}^{\circ}\cap R^{*})R_{*}=(E\cap E^{\circ})\cap (E^{\circ}R_{*}^{\circ}\cap R^{*})R_{*} \\
	& \stackrel{\ref{Modular Law*}}{\supseteq} & (E\cap E^{\circ})\cap E^{\circ}\cap R^{*}R_{*}\supseteq (E\cap E^{\circ})\cap E^{\circ}\cap E \\
	& = & E\cap E^{\circ}
\end{eqnarray*}	
where for establishing the first and second inclusions we used the modular law in $\mathrm{Rel}(\ca{E})$.
\end{proof}

Given a morphism $R_{*}\colon(X,E)\to (Y,F)\in\ca{E}_{ex/reg}$, we call the relation $\mathfrak{gr}(R_{*})$ defined above the \emph{graph} of $R_{*}$. Observe furthermore that $\mathfrak{gr}(R_{*})$ satisfies the following two basic equalities: 
\begin{displaymath}
F\circ\mathfrak{gr}(R_{*})=R_{*} 
\end{displaymath}	
\begin{displaymath}
\mathfrak{gr}(R_{*})^{\circ}\circ F=R^{*}
\end{displaymath}
Indeed, on one hand clearly $F\circ\mathfrak{gr}(R_{*})\subseteq FR_{*}=R_{*}$. On the other hand we have
\begin{align*}
F\circ\mathfrak{gr}(R_{*})& \supseteq R_{*}R^{*}\mathfrak{gr}(R_{*})=R_{*}R^{*}(R_{*}\cap (R^{*})^{\circ})\\
 &\stackrel{\ref{Modular Law}}{\supseteq} R_{*}(R^{*}R_{*}\cap E^{\circ})\supseteq R_{*}(E\cap E^{\circ})\supseteq R_{*}
\end{align*}
The second equality follows in a similar fashion. 

In fact, these two equalities characterize $\mathfrak{gr}(R_{*})$ in the following sense: if the morphism $\Phi\colon(X,E)\to (Y,F)$ is a map in $Q(\ca{E})$ with $F\Phi=R_{*}$ and $\Phi^{\circ}F=R^{*}$, then $\Phi=\mathfrak{gr}(R_{*})$. Indeed,
\begin{displaymath} 
\Phi\subseteq (F\cap F^{\circ})\Phi\subseteq F\Phi\cap F^{\circ}\Phi=F\Phi\cap(\Phi^{\circ}F)^{\circ}=R_{*}\cap(R^{*})^{\circ}=\mathfrak{gr}(R_{*})
\end{displaymath}
and so $\Phi=\mathfrak{gr}(R_{*})$ because $Q(\ca{E})$ is an allegory and hence the inclusion of maps is discrete.

Finally, the assignment $R_{*}\mapsto \mathfrak{gr}(R_{*})$ is functorial. Given $R_{*}\colon(X,E)\to (Y,F)$ and $S_{*}\colon(Y,F)\to(Z,G)$ we have 
\begin{displaymath}
G(\mathfrak{gr}(S_{*})\mathfrak{gr}(R_{*}))=(G\mathfrak{gr}(S_{*}))\mathfrak{gr}(R_{*})=S_{*}\mathfrak{gr}(R_{*})=S_{*}F\mathfrak{gr}(R_{*})=S_{*}R_{*} 
\end{displaymath}
\begin{displaymath}
(\mathfrak{gr}(S_{*})\mathfrak{gr}(R_{*}))^{\circ}G=\mathfrak{gr}(R_{*})^{\circ}\mathfrak{gr}(S_{*})^{\circ}G=\mathfrak{gr}(R_{*})^{\circ}S^{*}=\mathfrak{gr}(R_{*})^{\circ}FS^{*}=R^{*}S^{*}, 
\end{displaymath}
so we conclude by the above observation that $\mathfrak{gr}(S_{*})\mathfrak{gr}(R_{*})=\mathfrak{gr}(S_{*}R_{*})$. Also, $\mathfrak{gr}(1_{(X,E)*})=1_{(X,E)*}\cap (1_{(X,E)}^{*})^{\circ}=E\cap E^{\circ}$ and the latter is an identity morphism in $\mathrm{Map}(Q(\ca{E}))$.

\begin{remark}(on notation)
	Henceforth, to ease the notation, we shall often denote the graph $\mathfrak{gr}(R_{*})$ of a morphism $R_{*}\colon(X,E)\to (Y,E)\in\ca{E}_{ex/reg}$ simply by $R$, i.e. we will just drop the lower asterisk. This shall not present too much risk for confusion as $R_{*}$ will always have appeared before $R$.
\end{remark}
\vspace{3mm}

We now begin our work on proving that $\ca{E}_{ex/reg}$ is indeed the desired exact completion of the regular category $\ca{E}$. This is broken down into a sequence of more bite-sized pieces. First, we establish a fundamental result asserting the existence of certain canonical representations for morphisms in the bicategories $Q_{w}(\ca{E})$ and $Q(\ca{E})$. Following this, we repeatedly employ this representation to establish step by step that $\ca{E}_{ex/reg}$ has the desired finite limit and exactness properties making it an exact category. The arguments here are essentially motivated by the description of these properties in terms of the calculus of relations. Finally, we show that we have indeed constructed the exact completion by establishing the relevant universal property.

To begin with, we record a small lemma concerning jointly order-monomorphic pairs of morphisms in $\ca{E}_{ex/reg}$.

\begin{lemma}
	If $\xy\xymatrix{(Y,F) & (X,E)\ar[l]_{R_{*}}\ar[r]^{S_{*}} & (Z,G)}\endxy$ is a pair of morphisms in the category $\ca{E}_{ex/reg}$ such that $R^{*}R_{*}\cap S^{*}S_{*}=E$, then this pair is jointly order-monomorphic in $\ca{E}_{ex/reg}$.
\end{lemma}
\begin{proof}
	Assume that $R^{*}R_{*}\cap S^{*}S_{*}=E$ and let $H_{*},K_{*}\colon(A,T)\to (X,E)\in\ca{E}_{ex/reg}$ be such that $R_{*}H_{*}\leq R_{*}K_{*}$ and $S_{*}H_{*}\leq S_{*}K_{*}$ i.e. $R_{*}H_{*}\supseteq R_{*}K_{*}$ and $S_{*}H_{*}\supseteq S_{*}K_{*}$. Then we have
	\begin{eqnarray*}
		K_{*}H^{*}& = & EK_{*}H^{*}=(R^{*}R_{*}\cap S^{*}S_{*})K_{*}H^{*}\subseteq R^{*}R_{*}K_{*}H^{*}\cap S^{*}S_{*}K_{*}H^{*}\\
		& \subseteq & R^{*}R_{*}H_{*}H^{*}\cap S^{*}S_{*}H_{*}H^{*}\subseteq R^{*}R_{*}\cap S^{*}S_{*},\\
	\end{eqnarray*}
	hence $K_{*}H^{*}\subseteq E$. But recall that by definition of morphisms we have an adjunction $H_{*}\dashv H^{*}$ in $Q_{w}(\ca{E})$. Thus, $K_{*}H^{*}\subseteq E\iff K_{*}\subseteq H_{*}$, so that we obtain $H_{*}\leq K_{*}$.
\end{proof}

The result that follows will be of central importance in establishing all the desired properties of $\ca{E}_{ex/reg}$ throughout the remainder of this section. It says that any morphism of $Q(\ca{E})$ (hence also of $Q_{w}(\ca{E})$) can be expressed in a suitable way via morphisms of $\ca{E}_{ex/reg}$. This should be compared to the fact that in any regular category $\ca{C}$ every relation $R\colon X\looparrowright Y$ can be written as $R=gf^{\circ}$, where $f\colon Z\to X$ and $g\colon Z\to Y$ are morphisms in $\ca{C}$ with $f^{*}f_{*}\cap g^{*}g_{*}=I_{Z}$. Actually, our goal is to show in the end that it is \emph{precisely} this fact, since we will prove that $Q(\ca{E})$ is exactly the bicategory of relations of $\ca{E}_{ex/reg}$, while $Q_{w}(\ca{E})$ will be that of weakening-closed relations.

\begin{proposition}\label{Tabulations Existence}
	Let $\Phi\colon(X,E)\to (Y,F)$ be a morphism of $Q(\ca{E})$. Then there exists a pair of morphisms $\xy\xymatrix{(X,E) & (Z,T)\ar[l]_{R_{0*}}\ar[r]^{R_{1*}} & (Y,F)}\endxy$ in $\ca{E}_{ex/reg}$ such that 
	\begin{enumerate}
		\item $\Phi=R_{1}R_{0}^{\circ}$.
		\item $R_{0}^{*}R_{0*}\cap R_{1}^{*}R_{1*}=T$.
	\end{enumerate}	
	Moreover, any pair $(R_{0*},R_{1*})$ with these two properties has the following universal property: 
	
	Given any morphisms	$\xy\xymatrix{(X,E) & (C,G)\ar[l]_{S_{0*}}\ar[r]^{S_{1*}} & (Y,F)}\endxy$ in $\ca{E}_{ex/reg}$ such that $S_{1}S_{0}^{\circ}\subseteq\Phi$, there exists a unique morphism $H_{*}\colon(C,G)\to (Z,T)\in\ca{E}_{ex/reg}$ with $R_{0*}H_{*}=S_{0*}$ and $R_{1*}H_{*}=S_{1*}$.	
	\begin{proof}
		Suppose that $\Phi$ is represented by the $\nc{ff}$-morphism $\langle r_{0},r_{1}\rangle\colon Z\rightarrowtail X\times Y$ in $\ca{E}$. We set $T\coloneqq r_{0}^{\circ}Er_{0}\cap r_{1}^{\circ}Fr_{1}$ and $R_{0*}\coloneqq Er_{0}$, $R_{1*}\coloneqq Fr_{1}$. The relations $r_{0}^{\circ}Er_{0}$ and $r_{1}^{\circ}Fr_{1}$ are inverse images along $r_{0},r_{1}$ respectively of the congruences $E,F$, hence are themselves congruences. Thus, so is their intersection $T$.
		
		Also, we claim that that $R_{0*}$ and $R_{1*}$ as defined above are morphisms\newline 	$\xy\xymatrix{(X,E) & (Z,T)\ar[l]_{R_{0*}}\ar[r]^{R_{1*}} & (Y,F)}\endxy$ in $\ca{E}_{ex/reg}$. Let's check this for $R_{0*}$: first, we have $ER_{0*}=EEr_{0}=Er_{0}=R_{0*}$. Furthermore, 
		\begin{displaymath}
		R_{0*}T=Er_{0}(r_{0}^{\circ}Er_{0}\cap r_{1}^{\circ}Fr_{1})\subseteq Er_{0}r_{0}^{\circ}Er_{0}\subseteq E\Delta_{X}Er_{0}=EEr_{0}=Er_{0}=R_{0*},
		\end{displaymath} 
		hence $R_{0*}T=R_{0*}$. So $R_{0*}$ is at least a morphism in $Q_{w}(\ca{E})$. To show that it  is actually a map, define $R_{0}^{*}\coloneqq r_{0}^{\circ}E$. We then similarly have 
		\begin{displaymath}
		TR_{0}^{*}=(r_{0}^{\circ}Er_{0}\cap r_{1}^{\circ}Fr_{1})r_{0}^{\circ}E\subseteq r_{0}^{\circ}Er_{0}r_{0}^{\circ}E\subseteq r_{0}^{\circ}EE=r_{0}^{\circ}E=R_{0}^{*}\implies TR_{0}^{*}=R_{0}^{*}
		\end{displaymath}
		and $R_{0}^{*}E=r_{0}^{\circ}EE=r_{0}^{\circ}E=R_{0}^{*}$. And finally, $R_{0}^{*}R_{0*}=r_{0}^{\circ}EEr_{0}=r_{0}^{\circ}Er_{0}\supseteq T$ and $R_{0*}R_{0}^{*}=Er_{0}r_{0}^{\circ}E\subseteq EE=E$.
		
		Similarly, it follows that $R_{1*}$ is a morphism $(Z,T)\to (Y,F)\in\ca{E}_{ex/reg}$ whose right adjoint in $Q_{w}(\ca{E})$ is $R_{1}^{*}\coloneqq r_{1}^{\circ}F$.
		
		Just by the definitions, we have 
		\begin{displaymath}
		R_{0}^{*}R_{0*}\cap R_{1}^{*}R_{1*}=r_{0}^{\circ}EEr_{0}\cap r_{1}^{\circ}FFr_{1}=r_{0}^{\circ}Er_{0}\cap r_{1}^{\circ}Fr_{1}=T.
		\end{displaymath}
		In addition, $R_{0}=R_{0*}\cap (R_{0}^{*})^{\circ}=Er_{0}\cap(r_{0}^{\circ}E)^{\circ}=Er_{0}\cap E^{\circ}r_{0}\stackrel{\ref{Map distributivity}}{=}(E\cap E^{\circ})r_{0}$ and similarly $R_{1}=(F\cap F^{\circ})r_{1}$, so that
		\begin{displaymath}
		R_{1}R_{0}^{\circ}=(F\cap F^{\circ})r_{1}r_{0}^{\circ}(E\cap E^{\circ})=(F\cap F^{\circ})\Phi(E\cap E^{\circ})=\Phi.
		\end{displaymath} 
		where the last equality holds because $\Phi\in Q(\ca{E})$.
		\vspace{2mm}
		
		Next, we have to prove the stated universality property, so let\newline $\xy\xymatrix{(X,E) & (C,G)\ar[l]_{S_{0*}}\ar[r]^{S_{1*}} & (Y,F)}\endxy$ in $\ca{E}_{ex/reg}$ be such that $S_{1}S_{0}^{\circ}\subseteq\Phi$. 
		
		Set $H_{*}\coloneqq R_{0}^{*}S_{0*}\cap R_{1}^{*}S_{1*}$ and $H^{*}\coloneqq S_{0}^{*}R_{0*}\cap S_{1}^{*}R_{1*}$. Since, both $H_{*}$, $H^{*}$ are binary intersections of compositions of morphisms in $Q_{w}(\ca{E})$, it is immediate that they are both themselves morphisms in that bicategory. We will show that $H_{*}$ is moreover a map with $H^{*}$ as its right adjoint.
		
		First of all, we have that
		\begin{displaymath}
		H_{*}H^{*}\subseteq R_{0}^{*}S_{0*}S_{0}^{*}R_{0*}\cap R_{1}^{*}S_{1*}S_{1}^{*}R_{1*}\subseteq R_{0}^{*}ER_{0*}\cap R_{1}^{*}FR_{1*}=R_{0}^{*}R_{0*}\cap R_{1}^{*}R_{1*}=T
		\end{displaymath}
		For the other inclusion we argue as follows:
		\begin{eqnarray*}
			H^{*}H_{*} & = & (S_{0}^{*}R_{0*}\cap S_{1}^{*}R_{1*})(R_{0}^{*}S_{0*}\cap R_{1}^{*}S_{1*})\supseteq (S_{0}^{\circ}R_{0}\cap S_{1}^{\circ}R_{1})(R_{0}^{\circ}S_{0}\cap R_{1}^{\circ}S_{1}) \\
			& = & (S_{0}^{\circ}R_{0}\cap S_{1}^{\circ}R_{1})((R_{0}^{\circ}S_{0}\cap R_{1}^{\circ}S_{1})(G\cap G^{\circ})\cap R_{0}^{\circ}S_{0}) \\
			& \stackrel{\ref{Modular Law}}{\supseteq} & (G\cap G^{\circ})\cap (S_{0}^{\circ}R_{0}\cap S_{1}^{\circ}R_{1})R_{0}^{\circ}S_{0} \\
			& = & (G\cap G^{\circ})\cap ((G\cap G^{\circ})S_{0}^{\circ}R_{0}\cap S_{1}^{\circ}R_{1})R_{0}^{\circ}S_{0} \\
			& \stackrel{\ref{Modular Law*}}{\supseteq} & (G\cap G^{\circ})\cap (G\cap G^{\circ})\cap S_{1}^{\circ}R_{1}R_{0}^{\circ}S_{0}
		\end{eqnarray*}	
		But now, using adjunction properties in $Q(\mathcal{E})$ together with the assumption that $S_{1}S_{0}^{\circ}\subseteq\Phi$, we observe that 
		\begin{displaymath}
		S_{1}S_{0}^{\circ}\subseteq\Phi=R_{1}R_{0}^{\circ}\implies S_{1}\subseteq R_{1}R_{0}^{\circ}S_{0}\implies G\cap G^{\circ}\subseteq S_{1}^{\circ}R_{1}R_{0}^{\circ}S_{0},
		\end{displaymath}
		from which we deduce $H^{*}H_{*}\supseteq (G\cap G^{\circ})\cap S_{1}^{\circ}R_{1}R_{0}^{\circ}S_{0}=G\cap G^{\circ}$. Then, finally, $H^{*}H_{*}G\supseteq (G\cap G^{\circ})G$, which implies $H^{*}H_{*}\supseteq G$.
		
		Thus, $H_{*}$ is indeed a morphism in $\ca{E}_{ex/reg}$. Furthermore,
		\begin{displaymath}
		R_{0*}H_{*}=R_{0*}(R_{0}^{*}S_{0*}\cap R_{1}^{*}S_{1*})\subseteq R_{0*}R_{0}^{*}S_{0*}\subseteq ES_{0*}=S_{0*}
		\end{displaymath}
		\begin{displaymath}
		R_{0*}H_{*}=R_{0*}(R_{0}^{*}S_{0*}\cap R_{1}^{*}S_{1*})\supseteq R_{0}(R_{0}^{\circ}S_{0}\cap R_{1}^{\circ}S_{1})\stackrel{\ref{Modular Law}}{\supseteq} S_{0}\cap R_{0}R_{1}^{\circ}S_{1}=S_{0}
		\end{displaymath}
		where for the last inclusion we made use of the fact that $S_{1}S_{0}^{\circ}\subseteq\Phi=R_{1}R_{0}^{\circ}$ implies $S_{0}\subseteq R_{0}R_{1}^{\circ}S_{1}$ by the adjunction property in $Q(\ca{E})$. 
		Now we have $ER_{0*}H_{*}\supseteq ES_{0}$, which implies $R_{0*}H_{*}\supseteq S_{0*}$. Thus, we deduce that $R_{0*}H_{*}=S_{0*}$ and similarly one obtains $R_{1*}H_{*}=S_{1*}$. Finally, uniqueness is clear because, as we've proved in the previous lemma, the equality $R_{0}^{*}R_{0*}\cap R_{1}^{*}R_{1*}=T$ implies that $R_{0*},R_{1*}$ are jointly order-monomorphic in $\ca{E}_{ex/reg}$.
	\end{proof}
	
\end{proposition}
\vspace{5mm}

A pair of morphisms $(R_{0*},R_{1*})$ in $\ca{E}_{ex/reg}$ with the properties in \ref{Tabulations Existence} will be called a \emph{tabulation} of $\Phi\colon(X,E)\to (Y,F)\in Q(\ca{E})$. This terminology is borrowed from the theory of allegories. Note incidentally that the latter theory could have been directly applied to $Q(\ca{E})$, but this approach would not work for us. The reason for that is that we can not identify the morphisms of the would be completion $\ca{E}_{ex/reg}$ as maps in this allegory, but rather only in $Q_{w}(\ca{E})$. Thus, we need both of these bicategories at the same time: $Q_{w}(\ca{E})$ to identify the morphisms of $\ca{E}_{ex/reg}$ and $Q(\ca{E})$ to express the existence of tabulations and perform calculations more freely.

Furthermore, since as we've noted earlier every morphism $\Phi\colon(X,E)\to (Y,F)\in Q_{w}(\ca{E})$ can be considered as a morphism in $Q(\ca{E})$, we also have tabulations for morphisms of $Q_{w}(\ca{E})$. In this case the inclusion $S_{1}S_{0}^{\circ}\subseteq\Phi$ in the universal property of tabulations is equivalent to $S_{1*}S_{0}^{*}\subseteq\Phi$. Indeed, $S_{1}S_{0}^{\circ}\subseteq\Phi$ implies $FS_{1}S_{0}^{\circ}E\subseteq F\Phi E$, which is to say $S_{1*}S_{0}^{*}\subseteq\Phi$. Similarly, for the tabulation $(R_{0*},R_{1*})$ we have $\Phi=R_{1}R_{0}^{\circ}=R_{1*}R_{0}^{*}$.

As we've already claimed, the existence of tabulations will be a fundamental tool for establishing results about $\ca{E}_{ex/reg}$. As a first example of this, we can now completely characterize what it means for a pair of morphisms to be jointly order-monomorphic in $\ca{E}_{ex/reg}$.

\begin{corollary}
	A pair of morphisms $\xy\xymatrix{(Y,F) & (X,E)\ar[l]_{R_{*}}\ar[r]^{S_{*}} & (Z,G)}\endxy$ is jointly order-monomorphic in $\ca{E}_{ex/reg}$ if and only if $R^{*}R_{*}\cap S^{*}S_{*}=E$.
\end{corollary}
\begin{proof}
	We've already proven sufficiency earlier, so assume conversely that $R_{*},S_{*}$ are jointly order-monomorphic. By the proposition above, there exists a tabulation \\ $\xy\xymatrix{(X,E) & (A,T)\ar[l]_{U_{*}}\ar[r]^{V_{*}} & (X,E)}\endxy$ for the morphism $R^{*}R_{*}\cap S^{*}S_{*}\in Q_{w}(\mathcal{E})$. Then we have
	\begin{displaymath}
	V_{*}U^{*}=R^{*}R_{*}\cap S^{*}S_{*}\subseteq R^{*}R_{*}\implies R_{*}V_{*}U^{*}\subseteq R_{*}\implies R_{*}V_{*}\subseteq R_{*}U_{*}
	\end{displaymath}
	and similarly we obtain $S_{*}V_{*}\subseteq S_{*}U_{*}$. Thus, we have $R_{*}V_{*}\geq R_{*}U_{*}$ and $S_{*}V_{*}\geq S_{*}U_{*}$ and hence $V_{*}\geq U_{*}$, which is to say that $V_{*}\subseteq U_{*}$. But now we have
	\begin{displaymath}
	R^{*}R_{*}\cap S^{*}S_{*}=V_{*}U^{*}\subseteq U_{*}U^{*}\subseteq E
	\end{displaymath}
	Since the reverse inclusion always holds, we conclude that $R^{*}R_{*}\cap S^{*}S_{*}=E$.
\end{proof}

Before beginning to prove the basic finite limit and exactness properties of $\ca{E}_{ex/reg}$ we will need some information on $\nc{ff}$ and $\nc{so}$-morphisms therein.

\begin{lemma}\label{ff,so,isos in ex/reg}
	Let $R_{*}\colon(X,E)\to (Y,F)$ be a morphism in $\mathcal{E}_{ex/reg}$. Then:
	\begin{enumerate}
		\item $R_{*}$ is an $\nc{ff}$-morphism in $\ca{E}_{ex/reg}$ if and only if $R^{*}R_{*}=E$.
		\item $R_{*}$ is an iso if and only if $R^{*}R_{*}=E$ and $R_{*}R^{*}=F$.
		\item If $R_{*}R^{*}=F$, then $R_{*}$ is an $\nc{so}$-morphism in $\ca{E}_{ex/reg}$.
	\end{enumerate}
\end{lemma}
\begin{proof}
	\begin{enumerate}
		\item This follows immediately from the previous corollary.
		
		\item Clear.
		
		\item Consider a commutative square in $\ca{E}_{ex/reg}$ as below, where $M_{*}$ is an $\nc{ff}$-morphism.
		\begin{center}
			\hfil
			\xy\xymatrix{(X,E)\ar[r]^{R_{*}}\ar[d]_{V_{*}} & (Y,F)\ar[d]^{S_{*}} \\
				(Z,G)\ar@{>->}[r]_{M_{*}} & (W,H)}\endxy
			\hfil 
		\end{center}
		By (1.), we know that $M^{*}M_{*}=G$. Set $P_{*}\coloneqq V_{*}R^{*}$. First, we claim that $P_{*}$ is a morphism $(Y,F)\to(Z,G)$ in $\ca{E}_{ex/reg}$ with $P^{*}=R_{*}V^{*}$. Indeed, we have $P^{*}P_{*}=R_{*}V^{*}V_{*}R^{*}\supseteq R_{*}R^{*}=F$. Observe also that $P_{*}=M^{*}S_{*}$, because $M_{*}V_{*}=S_{*}R_{*}\implies M^{*}M_{*}V_{*}=M^{*}S_{*}R_{*}\implies V_{*}=M^{*}S_{*}R_{*}\implies V_{*}R^{*}=M^{*}S_{*}R_{*}R^{*}=M^{*}S_{*}$. Then we can argue that $P_{*}P^{*}=M^{*}S_{*}R_{*}V^{*}=M^{*}M_{*}V_{*}V^{*}=V_{*}V^{*}\subseteq G$.
		
		Finally, clearly $P_{*}R_{*}=M^{*}S_{*}R_{*}=M^{*}M_{*}V_{*}=V_{*}$ and also $M_{*}P_{*}=M_{*}V_{*}R^{*}=S_{*}R_{*}R^{*}=S_{*}$.
	\end{enumerate}
\end{proof}

\begin{proposition}\label{finite limits in ex/reg}
	$\ca{E}_{ex/reg}$ has finite limits and $\Gamma\colon\ca{E}\to\ca{E}_{ex/reg}$ preserves them.
\end{proposition}
\begin{proof}
	Let's first construct the inserter of a pair $\xy\xymatrix{(X,E)\ar@<1ex>[r]^{R_{*}}\ar@<-1ex>[r]_{S_{*}} & (Y,F)}\endxy$. To this end, consider a tabulation $\xy\xymatrix{(X,E) & (A,T)\ar[l]_{\Phi_{0*}}\ar[r]^{\Phi_{1*}} & (X,E)}\endxy$ of $S^{*}R_{*}\cap (E\cap E^{\circ})$ as a morphism $(X,E)\to(X,E)$ in $Q(\ca{E})$. Then we  observe that
	\begin{eqnarray*}
		\Phi_{1} & = & \Phi_{1}(T\cap T^{\circ})=\Phi_{1}(\Phi_{0}^{\circ}\Phi_{0}\cap\Phi_{1}^{\circ}\Phi_{1})\subseteq\Phi_{1}\Phi_{0}^{\circ}\Phi_{0}=(S^{*}R_{*}\cap E\cap E^{\circ})\Phi_{0} \\
		& \subseteq & (E\cap E^{\circ})\Phi_{0}=\Phi_{0}
	\end{eqnarray*}	
	and hence $\Phi_{1}=\Phi_{0}$ because the inclusion of maps in $Q(\ca{E})$ is discrete. So we have $\Phi_{1*}=\Phi_{0*}$, which we henceforth denote simply by $\Phi_{*}$. To prove that $\Phi_{*}\colon(A,T)\to (X,E)$ is the inserter of $(R_{*},S_{*})$ it now suffices, due to the universal property of tabulations, to show that for every $H_{*}\colon(Z,G)\to (X,E)$ we have $R_{*}H_{*}\leq S_{*}H_{*}$ if and only if $HH^{\circ}\subseteq S^{*}R_{*}\cap E\cap E^{\circ}$. Indeed, we have
	\begin{eqnarray*}
		R_{*}H_{*}\leq S_{*}H_{*} & \iff &  S_{*}H_{*}\subseteq R_{*}H_{*}\iff H_{*}\subseteq S^{*}R_{*}H_{*}\iff H_{*}H^{*}\subseteq S^{*}R_{*}\\
		& \iff & HH^{\circ}\subseteq S^{*}R_{*}\iff HH^{\circ}\subseteq S^{*}R_{*}\cap E\cap E^{\circ}
	\end{eqnarray*}
	
	Next, let us construct the product of a pair of objects $(X,E)$ and $(Y,F)$ in $\ca{E}_{ex/reg}$. Observe that the maximal relation $X\looparrowright Y$ given by the product is clearly a morphism $(X,E)\to (Y,F)$ in $Q(\ca{E})$. Therefore, by \ref{Tabulations Existence} it has a tabulation. In fact, looking back at the proof of the latter proposition we can easily see that the tabulation thus constructed is given by the pair of morphisms $\begin{tikzcd}[column sep=.75in](X,E) & (X\times Y,E\times F)\ar[l,"\Pi_{(X,E)*}"']\ar[r,"\Pi_{(Y,F)*}"] & (Y,F)\end{tikzcd}$, where $\Pi_{(X,E)*}=E\pi_{X}$ and $\Pi_{(Y,F)*}=F\pi_{Y}$. In any case, the universal property of tabulations gives precisely the universal property of a product diagram in $\ca{E}_{ex/reg}$.
	
	Finally, it is easy to check that $\Gamma 1$ is a terminal object for $\ca{E}_{ex/reg}$. 
\end{proof}

Even though the above proposition tells us that $\ca{E}_{ex/reg}$ inherits all finite weighted limits from $\ca{E}$, we shall need some more specific information as well, namely on the construction of comma squares and pullbacks.

Consider any morphisms $R_{*}\colon(X,E)\to(Z,G)$ and $S_{*}\colon(Y,F)\to(Z,G)$ in $\ca{E}_{ex/reg}$. First, we construct the comma square below:
\begin{center}
	\hfil
	\xy\xymatrix{(W,T)\ar[d]_{P_{0*}}\ar[r]^{P_{1*}}\ar@{}[dr]|{\leq} & (Y,F)\ar[d]^{S_{*}} \\
		(X,E)\ar[r]_{R_{*}} & (Z,G)}\endxy
	\hfil 
\end{center}
For this, we take $(P_{0*},P_{1*})$ to be a tabulation of $S^{*}R_{*}\colon(X,E)\to(Y,F)\in Q_{w}(\ca{E})$. To prove that this square is indeed a comma, it suffices to prove that, given any $U_{*}\colon(A,H)\to(X,E)$ and $V_{*}\colon(A,H)\to(Y,F)$, we have $V_{*}U^{*}\subseteq S^{*}R_{*}$ if and only if $R_{*}U_{*}\leq S_{*}V_{*}$. But indeed, using properties of adjunctions we have equivalences
\begin{displaymath}
R_{*}U_{*}\leq S_{*}V_{*}\iff S_{*}V_{*}\subseteq R_{*}U_{*}\iff V_{*}\subseteq S^{*}R_{*}U_{*}\iff V_{*}U^{*}\subseteq S^{*}R_{*}
\end{displaymath}

For pullbacks let us take now $(P_{0*},P_{1*})$ to be a tabulation of the relation $S^{\circ}R\in Q(\mathcal{E})$. Then we claim that the following square is a pullback.
\begin{center}
	\hfil
	\xy\xymatrix{(W,T)\ar[d]_{P_{0*}}\ar[r]^{P_{1*}} & (Y,F)\ar[d]^{S_{*}} \\
		(X,E)\ar[r]_{R_{*}} & (Z,G)}\endxy
	\hfil 
\end{center}
Indeed, in this case we have for any $U_{*}\colon(A,H)\to(X,E)$ and $V_{*}\colon(A,H)\to(Y,F)$ that
$R_{*}U_{*}=S_{*}V_{*}$ if and only if $RU=SV$, which in the allegory $Q(\ca{E})$ is equivalent to the inclusion $SV\subseteq RU$. By adjunction conditions, the latter is in turn equivalent to $SVU^{\circ}\subseteq R$ and then to $VU^{\circ}\subseteq S^{\circ}R$. Thus, the universal property of the pullback is identified with that of the tabulation.

Next, we prove that $\ca{E}_{ex/reg}$ admits the required factorization system.

\begin{proposition}\label{factorizations in ex/reg}
	$\ca{E}_{ex/reg}$ has ($\nc{so}$,$\nc{ff}$)-factorizations.
\end{proposition}
\begin{proof}
	Consider a morphism $R_{*}\colon(X,E)\to (Y,F)\in\ca{E}_{ex/reg}$. Then $RR^{\circ}\in Q(\ca{E})$ and so it admits a tabulation $\xy\xymatrix{(Y,F) & (Z,G)\ar[l]_{S_{0*}}\ar[r]^{S_{1*}} & (Y,F)}\endxy$. Since tautologically $R_{*}$ is such that $RR^{\circ}\subseteq S_{1}S^{\circ}_{0}$, there exists a unique $Q_{*}\colon(X,E)\to (Z,G)\in\ca{E}_{ex/reg}$ such that $S_{0*}Q_{*}=R_{*}=S_{1*}Q_{*}$.
	
	Now observe that in $Q(\ca{E})$ we have $S_{1}\subseteq S_{1}S_{0}^{\circ}S_{0}=RR^{\circ}S_{0}\subseteq S_{0}$ and so we deduce that $S_{1}=S_{0}$ and hence $S_{1*}=S_{0*}$. We denote this morphism now simply by $S_{*}$. Then we have $S^{*}S_{*}=G$ by the tabulation property and this tells us that $S_{*}$ is an $\nc{ff}$-morphism. It suffices now to show that $Q_{*}Q^{*}=G$, so that $Q_{*}$ will be an $\nc{so}$-morphism. For this we argue as follows:
	\begin{eqnarray*}
		SQQ^{\circ}S^{\circ} & = & RR^{\circ}=SS^{\circ}\implies \\
		S^{\circ}SQQ^{\circ}S^{\circ}S & = & S^{\circ}SS^{\circ}S \implies \\
		(G\cap G^{\circ})QQ^{\circ}(G\cap G^{\circ}) & = & G\cap G^{\circ}\implies \\
		QQ^{\circ} & = & G\cap G^{\circ} 
	\end{eqnarray*}
	Then $Q_{*}Q^{*}=GQQ^{\circ}G=G(G\cap G^{\circ})G=G$.
\end{proof}
\vspace{3mm}

\begin{remark}
	For a morphism $R_{*}\colon(X,E)\to (Y,F)\in\ca{E}_{ex/reg}$ we have $R_{*}R^{*}=F$ if and only if $RR^{\circ}=F\cap F^{\circ}$. We showed the ``if'' direction in the course of the above proof. For the converse, assume that $R_{*}R^{*}=F$ and argue as follows:
	\begin{align*}
		RR^{\circ} & =R(R^{\circ}(F\cap F^{\circ})\cap R^{*})\stackrel{\ref{Modular Law}}{\supseteq} (F\cap F^{\circ})\cap RR^{*}=(F\cap F^{\circ})\cap(R_{*}\cap (R^{*})^{\circ})R^{*}\\
		& \stackrel{\ref{Modular Law*}}{\supseteq} (F\cap F^{\circ})\cap R_{*}R^{*}\cap F^{\circ}= (F\cap F^{\circ})\cap F\cap F^{\circ}=F\cap F^{\circ}
	\end{align*}
\end{remark}

\begin{corollary}
	A morphism $R_{*}\colon(X,E)\to (Y,F)$ is an $\nc{so}$-morphism in $\ca{E}_{ex/reg}$ if and only if $R_{*}R^{*}=F$, if and only if $RR^{\circ}=F\cap F^{\circ}$. 
\end{corollary}
\begin{proof}
	If $R_{*}R^{*}=F$, then we know that $R_{*}$ is an $\nc{so}$-morphism by \ref{ff,so,isos in ex/reg}. In addition, by the above remark we know that $R_{*}R^{*}=F$ is equivalent to $RR^{\circ}=F\cap F^{\circ}$.
	
	Now assume that $R_{*}$ is an $\nc{so}$-morphism. From the proof of \ref{factorizations in ex/reg} we know that $R_{*}$ can be factored as $\begin{tikzcd}(X,E)\ar[r,two heads,"Q_{*}"] & (Z,G)\ar[r,tail,"S_{*}"] & (Y,F)\end{tikzcd}$, where $S_{*}$ is an $\nc{ff}$-morphism and $Q_{*}$ satisfies $Q_{*}Q^{*}=G$. Then $S_{*}$ is also an $\nc{so}$-morphism, since $R_{*}$ is such, and hence must be an iso. It then follows immediately that $R_{*}$ also satisfies $R_{*}R^{*}=F$.
\end{proof}

With this equational characterization of $\nc{so}$-morphisms in hand, we are now in a position to prove their stability under pullback.

\begin{proposition}
	$\nc{so}$-morphisms are stable under pullback in $\ca{E}_{ex/reg}$.
\end{proposition}
\begin{proof}
	Consider the following pullback square in $\ca{E}_{ex/reg}$ where we assume that $R_{*}$ is an $\nc{so}$-morphism, so that we have $R_{*}R^{*}=G$ or equivalently $RR^{\circ}= G\cap G^{\circ}$.
	\begin{center}
		\hfil
		\xy\xymatrix{(W,T)\ar[r]^{Q_{*}}\ar[d]_{P_{*}} & (Y,F)\ar[d]^{S_{*}} \\
			(X,E)\ar@{->>}[r]_{R_{*}} & (Z,G)}\endxy
		\hfil 
	\end{center}
	
	By the construction of pullbacks we know that $(P_{*},Q_{*})$ is a tabulation of $S^{\circ}R$, so that $QP^{\circ}=S^{\circ}R$. Thus, we have
	\begin{eqnarray*}
		QQ^{\circ} & = & Q(T\cap T^{\circ})Q^{\circ}=Q(P^{\circ}P\cap Q^{\circ}Q)Q^{\circ}\stackrel{\ref{Modular Law}}{\supseteq} (QP^{\circ}P\cap Q)Q^{\circ} \\
		& \stackrel{\ref{Modular Law*}}{\supseteq} & QP^{\circ}PQ^{\circ}\cap(F\cap F^{\circ})=S^{\circ}RR^{\circ}S\cap(F\cap F^{\circ}) \\
		& = & S^{\circ}(G\cap G^{\circ})S\cap(F\cap F^{\circ})=S^{\circ}S\cap(F\cap F^{\circ})=F\cap F^{\circ}
	\end{eqnarray*}	
	Hence, $QQ^{\circ}=F\cap F^{\circ}$ or equivalently $Q_{*}Q^{*}=F$ and hence $Q_{*}$ is an $\nc{so}$-morphism.
\end{proof}

Putting together what we have proved so far, we have the following.

\begin{corollary}
	$\ca{E}_{ex/reg}$ is a regular category and $\Gamma\colon\ca{E}\to\ca{E}_{ex/reg}$ is a fully order-faithful regular functor.
\end{corollary}
\vspace{3mm}

We next would like to prove that $\ca{E}_{ex/reg}$ is exact. To accomplish this we first make good on promises made much earlier. Namely, we identify $Q_{w}(\ca{E})$ as the bicategory of weakening-closed relations in $\ca{E}_{ex/reg}$ and, before that, $Q(\ca{E})$ as the bicategory of all relations in $\ca{E}_{ex/reg}$.

\begin{proposition}
	There is an equivalence $\mathrm{Rel}(\ca{E}_{ex/reg})\simeq Q(\ca{E})$.
\end{proposition}
\begin{proof}
	We will define a functor $\mathfrak{F}\colon\mathrm{Rel}$$(\ca{E}_{ex/reg})\to$ $Q(\ca{E})$ by letting it be the identity on objects and mapping a relation represented by any jointly order-monomorphic pair $\xy\xymatrix{(X,E) & (Z,T)\ar[l]_{R_{0*}}\ar[r]^{R_{1*}} & (Y,F)}\endxy$ in $\ca{E}_{ex/reg}$ to the morphism $R_{1}R_{0}^{\circ}\colon(X,E)\to(Y,F)\in Q(\ca{E})$.
	
	To show that this assignment is functorial, consider first the diagonal relation on the object $(X,E)$ in $\mathrm{Rel}$$(\ca{E}_{ex/reg})$, i.e. the relation represented by the jointly order-monomorphic pair $\xy\xymatrix{(X,E) & (X,E)\ar[l]_{1_{(X,E)*}}\ar[r]^{1_{(X,E)*}} & (X,E)}\endxy$. Then the image of this relation under $\mathfrak{F}$ is $1_{(X,E)}1_{(X,E)}^{\circ}=(E\cap E^{\circ})(E\cap E^{\circ})^{\circ}=E\cap E^{\circ}$ and so $\mathfrak{F}$ preserves identity morphisms.
	
	Next, we consider two relations $\mathscr{R},\mathscr{S}$ in $\ca{E}_{ex/reg}$, say represented respectively by the jointly order-monomorphic pairs $\xy\xymatrix{(X,E) & (A,T)\ar[l]_{R_{0*}}\ar[r]^{R_{1*}} & (Y,F)}\endxy$ and \\ $\xy\xymatrix{(Y,F) & (B,T')\ar[l]_{S_{0*}}\ar[r]^{S_{1*}} & (Z,G)}\endxy$. To calculate the composition of these two relations we form the following pullback square in $\ca{E}_{ex/reg}$
	\begin{center}
		\hfil
		\xy\xymatrix{(C,\Omega)\ar[r]^{\Pi_{1*}}\ar[d]_{\Pi_{0*}} & (B,T')\ar[d]^{S_{0*}} \\
			(A,T)\ar[r]_{R_{1*}} & (Y,F)}\endxy
		\hfil 
	\end{center}
	and then the image factorization of $\langle R_{0*}\Pi_{0*},S_{1*}\Pi_{1*}\rangle$, say
	\begin{displaymath}
	\xy\xymatrix@=5em{(C,\Omega)\ar@{->>}[r]^{Q_{*}} & (D,\Theta)\ar@{>->}[r]^(0.4){\langle U_{0*},U_{1*}\rangle} & (X,E)\times(Z,G)}\endxy
	\end{displaymath}
	By construction of pullbacks in $\ca{E}_{ex/reg}$ we know that $\Pi_{1}\Pi_{0}^{\circ}=S_{0}^{\circ}R_{1}$. Also, by definition $\mathfrak{F}$ maps the composition of the two relations to $\mathfrak{F}(\mathscr{S}\mathscr{R})=U_{1}U_{0}^{\circ}$. But now we have that 
	\begin{displaymath}
	U_{1}U_{0}^{\circ}=U_{1}QQ^{\circ}U_{0}^{\circ}=S_{1}\Pi_{1}\Pi_{0}^{\circ}R_{0}^{\circ}=S_{1}S_{0}^{\circ}R_{1}R_{0}^{\circ}=\mathfrak{F}(\mathscr{S})\mathfrak{F}(\mathscr{R})
	\end{displaymath}
	
	Finally, the fact that $\mathfrak{F}$ preserves the order of morphisms and is fully (order-) faithful is precisely the existence of tabulations proved in \ref{Tabulations Existence}. Thus, $\mathfrak{F}$ is an equivalence of bicategories.
\end{proof}
\vspace{3mm}

\begin{proposition}
	There is an equivalence $\mathrm{Rel}_{w}(\ca{E}_{ex/reg})\simeq Q_{w}(\ca{E})$.
\end{proposition}
\begin{proof}
	We define a functor $\mathfrak{F}\colon\mathrm{Rel}_{w}$$(\ca{E}_{ex/reg})\to$ $Q_{w}(\ca{E})$ exactly as in the proof of the previous proposition. Then the main observation to make here is the following: a relation $\mathscr{R}\colon(X,E)\looparrowright(Y,F)$ represented by the jointly order-monomorphic pair $\xy\xymatrix{(X,E) & (Z,T)\ar[l]_{R_{0*}}\ar[r]^{R_{1*}} & (Y,F)}\endxy$ in $\ca{E}_{ex/reg}$ is weakening-closed if and only if $R_{1*}R_{0}^{*}=R_{1}R_{0}^{\circ}$ as relations in $\ca{E}$.
	
	Recall that $\mathscr{R}$ is a weakening-closed relation precisely if $I_{(Y,F)}\mathscr{R}I_{(X,E)}=\mathscr{R}$ in $\mathrm{Rel}$$(\ca{E}_{ex/reg})$. To compute the composition $I_{(Y,F)}\mathscr{R}I_{(X,E)}$ one has to form the following diagram in $\ca{E}_{ex/reg}$ where the top square is a pullback and the bottom two are commas and then take the image factorization of the morphism $\langle U_{0*}W_{0*},V_{1*}W_{1*}\rangle$.
	\begin{center}
		\hfil 
		\xy\xymatrix{ & & (C,X)\ar[dl]_{W_{0*}}\ar[dr]^{W_{1*}} & &  \\
			& (A,\Phi)\ar[dl]_{U_{0*}}\ar[dr]^{U_{1*}} & & (B,\Psi)\ar[dl]_{V_{0*}}\ar[dr]^{V_{1*}} & \\
			(X,E)\ar[dr]_{1_{(X,E)*}}\ar@{}[rr]|{\leq} & & (Z,T)\ar[dl]_{R_{0*}}\ar[dr]^{R_{1*}}\ar@{}[rr]|{\leq} & & (Y,F)\ar[dl]^{1_{(Y,F)*}}  \\
			& (X,E) & & (Y,F) & }\endxy
		\hfil 
	\end{center}
	Note that by the various limit constructions in $\ca{E}_{ex/reg}$ we know that we must have $R_{0}^{*}=U_{1}U_{0}^{\circ}$, $R_{1*}=V_{1}V_{0}^{\circ}$ and $V_{0}^{\circ}U_{1}=W_{1}W_{0}^{\circ}$.
	
	If $I_{(Y,F)}\mathscr{R}I_{(X,E)}=\mathscr{R}$, then there is a factorization in $\ca{E}_{ex/reg}$
	\begin{displaymath}
	\xy\xymatrix@=5em{(C,X)\ar@{->>}[r]^{Q_{*}} & (Z,T)\ar@{>->}[r]^(0.4){\langle R_{0*},R_{1*}\rangle} & (X,E)\times(Y,F)}\endxy
	\end{displaymath}
	with $Q_{*}$ an $\nc{so}$-morphism, so that $QQ^{\circ}=T\cap T^{\circ}$. Then we have that
	\begin{displaymath}
	R_{1*}R_{0}^{*}=V_{1}V_{0}^{\circ}U_{1}U_{0}^{\circ}=V_{1}W_{1}W_{0}^{\circ}U_{0}^{\circ}=R_{1}QQ^{\circ}R_{0}^{\circ}=R_{1}R_{0}^{\circ}
	\end{displaymath}
	
	Conversely, assume that $R_{1*}R_{0}^{*}=R_{1}R_{0}^{\circ}$ and let us consider four morphisms $\xy\xymatrix{(X,E) & (C,G)\ar@<-1ex>[l]_{S_{0*}}\ar@<1ex>[l]^{U_{*}}\ar@<1ex>[r]^{S_{1*}}\ar@<-1ex>[r]_{V_{*}} & (Y,F)}\endxy$ such that $(S_{0*},S_{1*})$ factors through $(R_{0*},R_{1*})$ and $U_{*}\leq S_{0*}$ and $S_{1*}\leq V_{*}$. Then we respectively have $S_{1*}S_{0}^{*}\subseteq R_{1*}R_{0}^{*}$ and $U^{*}\subseteq S_{0}^{*}$ and $V_{*}\subseteq S_{1*}$. Hence, $V_{*}U^{*}\subseteq S_{1*}S_{0}^{*}\subseteq R_{1*}R_{0}^{*}=R_{1}R_{0}^{\circ}$ and so $(U_{*},V_{*})$ must also factor through $(R_{0*},R_{1*})$ by the universal property of tabulations.
	\vspace{3mm}
	
	With this observation in hand, one can run the same proof as in the previous proposition to show that $\mathfrak{F}\colon\mathrm{Rel}_{w}(\ca{E}_{ex/reg})\to Q_{w}(\ca{E})$ thus defined is an equivalence. The only point of minor difference is in the proof that $\mathfrak{F}$ preserves identity morphisms. For this, just recall that the identity on an object $(X,E)$ in $\mathrm{Rel}_{w}(\ca{E}_{ex/reg})$ is the relation given by the following comma square
	\begin{center}
		\hfil
		\xy\xymatrix{(A,T)\ar[r]^{R_{1*}}\ar[d]_{R_{0*}}\ar@{}[dr]|{\leq} & (X,E)\ar[d]^{1_{(X,E)*}} \\
			(X,E)\ar[r]_{1_{(X,E)*}} & (X,E)}\endxy
		\hfil 
	\end{center}
	so that by construction of commas we have that $R_{1}R_{0}^{\circ}=1_{(X,E)}^{*}1_{(X,E)*}=EE=E$.
\end{proof}
\vspace{3mm}

Now from this last proposition one can immediately deduce that $\ca{E}_{ex/reg}$ is indeed an exact category.

\begin{corollary}
	The category $\ca{E}_{ex/reg}$ is exact.
\end{corollary}
\begin{proof}
	Using the equivalence $\mathrm{Rel}_{w}(\ca{E}_{ex/reg})\simeq Q_{w}(\ca{E})$, we can see that a congruence on an object $(X,E)\in\ca{E}_{ex/reg}$ corresponds precisely to a congruence $R$ on the object $X\in\ca{E}$ with $R\supseteq E$. 
	
	Indeed, consider a congruence $\mathscr{R}$ on the object $(X,E)\in\ca{E}_{ex/reg}$, represented by the jointly order-monomorphic pair $\begin{tikzcd}(X,E) & (Y,F)\ar[l,"R_{0*}"']\ar[r,"R_{1*}"] & (X,E)\end{tikzcd}$. Consider the functor $\mathfrak{F}\colon\mathrm{Rel}_{w}(\ca{E}_{ex/reg})\to Q_{w}(\ca{E})$ providing the equivalence and let $R\coloneqq\mathfrak{F}(\mathscr{R})$. Since $\mathfrak{F}$ is an equivalence, we have that $\mathscr{R}\mathscr{R}\subseteq\mathscr{R}$ if and only if $RR\subseteq R$ in $Q_{w}(\ca{E})$, i.e. that transitivity of $\mathscr{R}$ is equivalent to the same property for $R$ as a relation in $\ca{E}$. Similarly, the inclusion $\mathscr{R}\supseteq I_{(X,E)}$ is equivalent to $R\supseteq\mathfrak{F}(I_{(X,E)})=E$. In particular, $R\supseteq I_{X}$ and so $R$ is a congruence on $X\in\ca{E}$.
	
	The idempotent morphism $\mathfrak{F}(\mathscr{R})=R\colon(X,E)\to(X,E)$ in $Q_{w}(\ca{E})$ now splits by construction, namely as $\begin{tikzcd}(X,E)\ar[r,"R"] & (X,R)\ar[r,"R"] & (X,E)\end{tikzcd}$. Thus, $\mathscr{R}$ splits as an idempotent in $\mathrm{Rel}_{w}$$(\ca{E}_{ex/reg})$ and hence we have shown that $\ca{E}_{ex/reg}$ is exact.
\end{proof}

It remains to prove that $\ca{E}_{ex/reg}$, or more precisely $\Gamma\colon\ca{E}\to\ca{E}_{ex/reg}$ satisfies the required universal property. Before doing this, we observe in the proposition that follows that every object of $\ca{E}_{ex/reg}$ appears as a quotient of a congruence coming from $\ca{E}$ in a canonical way.

\begin{proposition}
	For every object $(X,E)\in\ca{E}_{ex/reg}$ there exists an exact sequence $\xy\xymatrix{\Gamma E\ar@<1ex>[r]^{\Gamma e_{0}}\ar@<-1ex>[r]_{\Gamma e_{1}} & \Gamma X\ar@{->>}[r]^{E} & (X,E)}\endxy$, where $\langle e_{0},e_{1}\rangle\colon E\rightarrowtail X\times X$ is an $\nc{ff}$-morphism representing the congruence $E$ in $\ca{E}$.
\end{proposition}
\begin{proof}
	Observe that $E$ indeed defines a morphism $E_{*}\colon\Gamma X\to (X,E)$ in $\ca{E}_{ex/reg}$ which is in fact effective because $EI_{X}=E=EE$, $E^{*}E_{*}=EE=E\supseteq I_{X}$, $E_{*}E^{*}=EE=E$. Now it suffices to show that the square below is a comma square in $\ca{E}_{ex/reg}$.
	\begin{center}
		\hfil
		\xy\xymatrix{\Gamma E\ar[r]^{\Gamma e_{1}}\ar[d]_{\Gamma e_{0}}\ar@{}[dr]|{\leq} & \Gamma X\ar@{->>}[d]^{E} \\
			\Gamma X\ar@{->>}[r]_{E} & (X,E)}\endxy
		\hfil 
	\end{center}
	But by the construction of comma squares in $\ca{E}_{ex/reg}$ this is equivalent to having $(\Gamma e_{0})^{*}\Gamma e_{0}\cap(\Gamma e_{1})^{*}\Gamma e_{1}=I_{E}$ and $\Gamma e_{1}(\Gamma e_{0})^{\circ}=E^{*}E_{*}$ i.e. $e_{0}^{*}e_{0*}\cap e_{1}^{*}e_{1*}=I_{E}$ and $e_{1}e_{0}^{\circ}=EE$ as relations in $\ca{E}$, both of which hold.
\end{proof}
\vspace{3mm}

Now we at last come to the proof that $\ca{E}_{ex/reg}$ satisfies the universal property that exhibits it as the exact completion of the regular category $\ca{E}$. Before proceeding, let us make a couple of quick observations that will be used in the course of our calculations in the proof that follows below.

Consider an exact fork $\xy\xymatrix{E\ar@<1ex>[r]\ar@<-1ex>[r] & X\ar@{->>}[r]^{p} & P}\endxy$ in the regular category $\ca{E}$. Recall that in the calculus of relations this means that $E=p^{*}p_{*}$ and $p_{*}p^{*}=I_{P}$, where the second equality can equivalently be replaced by $pp^{\circ}=\Delta_{P}$. In addition, the kernel pair $p^{\circ}p$ of $p$ can be written as $p^{*}p_{*}\cap(p^{*}p_{*})^{\circ}$ and hence we also have $p^{\circ}p=E\cap E^{\circ}$  We now observe that the following equalities must hold:
\begin{itemize}
	\item $pE=p_{*}$.
	\item $Ep^{\circ}=p^{*}$.
	\item $p_{*}Ep^{*}=I_{P}$.
\end{itemize}
Indeed, for the first of these we have $pE\subseteq p_{*}E=p_{*}p^{*}p_{*}=p_{*}$ and also $pE=pp^{*}p_{*}\supseteq pp^{\circ}p_{*}=\Delta_{P}p_{*}=p_{*}$. The second one follows similarly. Finally, for the last one we have $p_{*}Ep^{*}=p_{*}p^{*}p_{*}p^{*}=I_{P}I_{P}=I_{P}$.
\vspace{1cm}

\begin{theorem}\label{Universal property of ex/reg}
	Let $F\colon\ca{E}\to\ca{F}$ be a regular functor with $\mathcal{F}$ an exact category. Then there is a unique (up to iso) regular functor $\overline{F}\colon\ca{E}_{ex/reg}\to\ca{F}$ such that $\overline{F}\circ\Gamma\cong F$.
\end{theorem}
\begin{proof}
	If $\overline{F}$ is to be regular, then by the previous proposition we must define it on any object $(X,E)\in\ca{E}_{ex/reg}$ as the following coinserter in $\ca{F}$
	\begin{center}
		\hfil
		\xy\xymatrix@=4em{F(E)\ar@<1ex>[r]^{Fe_{0}}\ar@<-1ex>[r]_{Fe_{1}} & FX\ar@{->>}[r]^{p_{(X,E)}} & \overline{F}(X,E)}\endxy
		\hfil
	\end{center}
	which exists because regular functors preserve congruences and $\ca{F}$ is exact.
	
	Also, given any morphism $R_{*}\colon(X,E)\to(Y,G)\in\ca{E}_{ex/reg}$, we have relations $F(R_{*})\colon FX\looparrowright FY$ and $F(R^{*})\colon FY\looparrowright FX$ in $\ca{F}$ satisfying the following equations:
	\begin{gather*}
	F(G)F(R_{*})F(E)=F(GR_{*}E)=F(R_{*}) \\
	F(E)F(R^{*})F(G)=F(ER^{*}G)=F(R^{*}) \\
	F(R^{*})F(R_{*})=F(R^{*}R_{*})\supseteq F(E) \\ 
	F(R_{*})F(R^{*})=F(R_{*}R^{*})\subseteq F(G) \\
	\end{gather*}
	where we used the fact that regular functors preserve the compositions and inclusions of relations. Thus, by \ref{morphisms between quotients of congruences} we can define $\overline{F}(R_{*})$ to be the uniquely associated morphism between quotients $\overline{F}(X,E)\to\overline{F}(Y,G)$. More explicitly, $\overline{F}(R_{*})$ is the morphism uniquely determined by the equality $\overline{F}(R_{*})_{*}=p_{(Y,G)*}F(R_{*})p_{(X,E)}^{*}$ in $\mathrm{Rel}(\ca{F})$. It is immediate that this defines a functor $\ca{E}_{ex/reg}\to\ca{F}$ and clearly $\overline{F}\circ\Gamma\cong F$. Note that by the discussion following \ref{morphisms between quotients of congruences} we also have $\overline{F}(R_{*})=p_{(Y,G)}(F(R_{*})\cap F(R^{*})^{\circ})p_{(X,E)}^{\circ}=p_{(Y,G)}F(R)p_{(X,E)}^{\circ}$ as relations in $\ca{F}$.
	
	Now let's show that $\overline{F}$ preserves finite limits. First, it is clear that it preserves the terminal object $\Gamma1$, since $\overline{F}\Gamma1=F1$ and $F$ preserves the terminal object. Second, the preservation of binary products follows from the fact that exact sequences in any regular category are stable under binary products. It suffices then to prove the preservation of inserters. So suppose that $\begin{tikzcd}(A,T)\ar[r,tail,"\Phi_{*}"] & (X,E)\ar[r,shift left=0.75ex,"R_{*}"]\ar[r,shift right=0.75ex,"S_{*}"'] & (Y,G)\end{tikzcd}$ is an inserter diagram in $\ca{E}_{ex/reg}$. Recall from \ref{finite limits in ex/reg} that by construction of inserters in $\ca{E}_{ex/reg}$ this means that $\Phi\Phi^{\circ}=S^{*}R_{*}\cap(E\cap E^{\circ})$ and $\Phi^{*}\Phi_{*}=T$ as relations in $\ca{E}$. Then first of all we have 
	\begin{eqnarray*}
	\overline{F}(\Phi_{*})^{*}\overline{F}(\Phi_{*})_{*}& = & p_{(A,T)*}F(\Phi^{*})p_{(X,E)}^{*}p_{(X,E)*}F(\Phi_{*})p_{(A,T)}^{*} \\
	       & = & p_{(A,T)*}F(\Phi^{*})F(E)F(\Phi_{*})p_{(A,T)}^{*} \\
	       & = & p_{(A,T)*}F(\Phi^{*}E\Phi_{*})p_{(A,T)}^{*} \\
	       & = & p_{(A,T)*}F(\Phi^{*}\Phi_{*})p_{(A,T)}^{*} \\
	       & = & p_{(A,T)*}F(T)p_{(A,T)}^{*}=I_{\overline{F}(A,T)}
	\end{eqnarray*}
	which tells us that $\overline{F}(\Phi_{*})$ is an $\nc{ff}$-morphism in $\ca{F}$. Second, we have the following sequence of calculations:
	\begin{align*}
		&\overline{F}(S_{*})^{*}\overline{F}(R_{*})_{*}\cap \Delta_{\overline{F}(X,E)}= \\
		& =  p_{(X,E)}p_{(X,E)}^{\circ}(\overline{F}(S_{*})^{*}\overline{F}(R_{*})_{*}\cap \Delta_{\overline{F}(X,E)})p_{(X,E)}p_{(X,E)}^{\circ} \\
		& =  p_{(X,E)}p_{(X,E)}^{\circ}(p_{(X,E)*}F(S^{*})p_{(Y,G)}^{*}p_{(Y,G)*}F(R_{*})p_{(X,E)}^{*}\cap\Delta_{\overline{F}(X,E)})p_{(X,E)}p_{(X,E)}^{\circ} \\
		   & =  p_{(X,E)}p_{(X,E)}^{\circ}(p_{(X,E)}F(E)F(S^{*})F(G)F(R_{*})F(E)p_{(X,E)}^{\circ}\cap\Delta_{\overline{F}(X,E)})p_{(X,E)}p_{(X,E)}^{\circ} \\
		   & =  p_{(X,E)}p_{(X,E)}^{\circ}(p_{(X,E)}F(ES^{*}GR_{*}E)p_{(X,E)}^{\circ}\cap\Delta_{\overline{F}(X,E)})p_{(X,E)}p_{(X,E)}^{\circ} \\
		   & =  p_{(X,E)}p_{(X,E)}^{\circ}(p_{(X,E)}F(S^{*}R_{*})p_{(X,E)}^{\circ}\cap\Delta_{\overline{F}(X,E)})p_{(X,E)}p_{(X,E)}^{\circ} \\
		   & \stackrel{\ref{Map distributivity},\ref{Map distributivity*}}{=}  p_{(X,E)}[p_{(X,E)}^{\circ}p_{(X,E)}F(S^{*}R_{*})p_{(X,E)}^{\circ}p_{(X,E)}\cap p_{(X,E)}^{\circ}p_{(X,E)}]p_{(X,E)}^{\circ} \\
		   & =  p_{(X,E)}[F(E\cap E^{\circ})F(S^{*}R_{*})F(E\cap E^{\circ})\cap F(E\cap E^{\circ})]p_{(X,E)}^{\circ} \\
		   & =  p_{(X,E)}(F(S^{*}R_{*})\cap F(E\cap E^{\circ}))p_{(X,E)}^{\circ} \\
		   & =  p_{(X,E)}F(S^{*}R_{*}\cap E\cap E^{\circ})p_{(X,E)}^{\circ} \\
		   & =  p_{(X,E)}F(\Phi\Phi^{\circ})p_{(X,E)}^{\circ} \\
		   & =  p_{(X,E)}F(\Phi)F(T\cap T^{\circ})F(\Phi^{\circ})p_{(X,E)}^{\circ} \\
		   & =  p_{(X,E)}F(\Phi)p_{(A,T)}^{\circ}p_{(A,T)}F(\Phi^{\circ})p_{(X,E)}^{\circ} \\
		   & =  \overline{F}(\Phi_{*})\overline{F}(\Phi_{*})^{\circ}
	\end{align*}

	These tell us that $\begin{tikzcd}\overline{F}(A,T)\ar[r,tail,"\overline{F}(\Phi_{*})"] & \overline{F}(X,E)\ar[r,shift left=0.75ex,"\overline{F}(R_{*})"]\ar[r,shift right=0.75ex,"\overline{F}(S_{*})"'] & \overline{F}(Y,G)\end{tikzcd}$ is an inserter diagram in $\ca{F}$.
	
	Next, consider an $\nc{so}$-morphism $R_{*}\colon(X,E)\twoheadrightarrow(Y,G)\in\ca{E}_{ex/reg}$. This means that $R_{*}R^{*}=G$ and then we have
	\begin{eqnarray*}
		\overline{F}(R_{*})_{*}\overline{F}(R_{*})^{*} & = & p_{(Y,G)*}F(R_{*})p_{(X,E)}^{*}p_{(X,E)*}F(R^{*})p_{(Y,G)}^{*} \\
		& = & p_{(Y,G)*}F(R_{*})F(E)F(R^{*})p_{(Y,G)}^{*} \\
		& = & p_{(Y,G)*}F(R_{*}R^{*})p_{(Y,G)}^{*}=p_{(Y,G)*}F(G)p_{(Y,G)}^{*} \\
		& = & I_{\overline{F}(Y,G)}
	\end{eqnarray*}
	from which we obtain that $\overline{F}(R_{*})$ is an $\nc{so}$-morphism in $\ca{F}$. Thus, we have proved that $\overline{F}$ is a regular functor.
	
	Finally, for any regular functor $H\colon\ca{E}_{ex/reg}\to\ca{F}$, for every object $(X,E)\in\ca{E}_{ex/reg}$ we must have an exact sequence $\xy\xymatrix{H\Gamma E\ar@<1ex>[r]\ar@<-1ex>[r] & H\Gamma X\ar@{->>}[r] & H(X,E)}\endxy$ in $\ca{F}$. If $H\Gamma\cong F$, this forces $H\cong\overline{F}$.
\end{proof}

\section{A Characterization of the Exact Completion and Priestley Spaces}

Having established the universal property of the exact completion, in this section we present a result which identifies the situation in which an exact category is the exact completion of a given regular category $\ca{E}$. More precisely, we will characterize the canonical functor $\Gamma\colon\ca{E}\to\ca{E}_{ex/reg}$ as the unique up to equivalence functor from $\ca{E}$ into an exact category which satisfies some simple properties.  This will in turn allow us to easily deduce some examples of categories which arise as exact completions of some familiar regular subcategory. 

The main example that we aim to cover here involves the category of \emph{Priestley} spaces. Indeed, the latter is regular as a $\nc{Pos}$-category and we prove that its exact completion is the category of \emph{compact ordered spaces} (or \emph{Nachbin} spaces). This provides an ordered version of the folklore result which identifies the category of compact Hausdorff spaces as the exact completion (in the ordinary sense) of the regular category of \emph{Stone} spaces (see e.g. \cite{Panagis & me}). 

But first, we need some preliminaries.
We will say that a functor $F\colon\ca{C}\to\ca{D}$ is \emph{order-faithful} if, for every $f,g\colon X\to Y\in\ca{C}$ we have $Ff\leq Fg\implies f\leq g$. In other words, $F$ is order-faithful if for every $X,Y\in\ca{C}$ the morphism $\ca{C}(X,Y)\to\ca{D}(FX,FY)$ is an $\nc{ff}$-morphism in $\nc{Pos}$. Note in particular that such a functor is faithful in the ordinary sense or, in more appropriate language, the underlying functor between ordinary categories is faithful. In fact, if $\ca{C}$ has inserters which are preserved by $F\colon\ca{C}\to\ca{D}$, then the two notions coincide.

The crux of the work now consists of establishing that certain properties of a regular functor $F\colon\ca{E}\to\ca{F}$ into an exact category $\ca{F}$ translate to corresponding properties of the induced $\overline{F}\colon\ca{E}_{ex/reg}\to\ca{F}$.

\begin{lemma}
	Let $F\colon\ca{E}\to\ca{F}$ be a regular functor with $\ca{F}$ an exact category. If $F$ is fully order-faithful, then $\overline{F}\colon\ca{E}_{ex/reg}\to\ca{F}$ is order-faithful.
\end{lemma}
\begin{proof}
	Let $R_{*},S_{*}\colon(X,E)\to(Y,G)\in\ca{E}_{ex/reg}$ be such that $\overline{F}(R_{*})\leq\overline{F}(S_{*})$ in $\ca{F}$. Then using the definition of $\overline{F}$ we have
	\begin{eqnarray*}
		\overline{F}(R_{*})_{*}\supseteq\overline{F}(S_{*})_{*} & \implies & p_{(Y,G)}^{*}\overline{F}(R_{*})_{*}p_{(X,E)*}\supseteq p_{(Y,G)}^{*}\overline{F}(S_{*})_{*}p_{(X,E)*} \\
		& \implies & p_{(Y,G)}^{*}p_{(Y,G)*}F(R_{*})p_{(X,E)}^{*}p_{(X,E)*}\supseteq \\
		& \supseteq & p_{(Y,G)}^{*}p_{(Y,G)*}F(S_{*})p_{(X,E)}^{*}p_{(X,E)*} \\
		& \implies & F(G)F(R_{*})F(E)\supseteq F(G)F(S_{*})F(E) \\
		& \implies & F(R_{*})\supseteq F(S_{*})
	\end{eqnarray*}
	But since $F$ is fully (order-) faithful, it reflects inclusions of subobjects and hence we obtain $R_{*}\supseteq S_{*}$, i.e. $R_{*}\leq S_{*}$ in $\ca{E}_{ex/reg}$.
\end{proof}

We introduce some further properties of functors that will be of interest. Our choice of terminology follows the literature of Categorical Logic (e.g.\cite{Makkai-Reyes}).

\begin{definition}
	A functor $F\colon\ca{C}\to\ca{D}$ is called \emph{covering} if, for every object $Y\in\ca{D}$, one can find an object $X\in\ca{C}$ and an effective  epimorphism $FX\twoheadrightarrow Y$.
	
	We say that $F$ is \emph{full on subobjects} if, for every $\nc{ff}$-morphism $B\rightarrowtail FX$ in $\ca{D}$, there exists an $\nc{ff}$-morphism $A\rightarrowtail X$ in $\ca{C}$ such that $FA\cong B$ in $\mathrm{Sub}_{\ca{D}}$$(FX)$.
\end{definition}

The following basic observation (even for ordinary categories) seems to not have appeared explicitly in the literature. Since we will need it below, we give its easy proof.

\begin{lemma}
	Let $F\colon\ca{C}\to\ca{D}$ be a regular functor between regular categories. If $F$ is full and covering, then it is full on subobjects.
\end{lemma}
\begin{proof}
	Consider an $\nc{ff}$-morphism $v\colon D\rightarrowtail FY$ in $\ca{D}$. Since $F$ is covering, there exists some $\nc{so}$-morphism $q\colon FX\twoheadrightarrow D$ in $\ca{D}$. Now consider the composition of the two, $\xy\xymatrix{FX\ar@{->>}[r]^{q} & D\ar@{>->}[r]^{v} & FY}\endxy$. Since $F$ is full, there is a morphism $f\colon X\to Y$ in $\ca{C}$ such that $Ff=vq$.
	
	Since $\ca{C}$ is regular, we can factor $f$ as an $\nc{so}$ followed by an $\nc{ff}$-morphism, say $f=\xy\xymatrix{X\ar@{->>}[r]^{p} & I\ar@{>->}[r]^{u} & Y}\endxy$. But $F$ is a regular functor, so $\xy\xymatrix{FX\ar@{->>}[r]^{Fp} & FI\ar@{>->}[r]^{Fu} & FY}\endxy$ is the ($\nc{so}$,$\nc{ff}$) factorization of $Ff=vq$. By uniqueness of such factorizations in the regular category $\ca{D}$, we deduce that $FI\cong D$ as subobjects of $FY$.
\end{proof}

\begin{proposition}
	Let $F\colon\ca{E}\to\ca{F}$ be a regular functor with $\ca{F}$ an exact category. If $F$ is fully order-faithful and covering, then $\overline{F}\colon\ca{E}_{ex/reg}\to\ca{F}$ is fully order-faithful and covering.
\end{proposition}
\begin{proof}
	We saw earlier that $F$ being fully order-faithful implies that $\overline{F}$ is order-faithful. Furthermore, it is immediate that $F$ being covering implies the same property for $\overline{F}$, since we have $\overline{F}\Gamma\cong F$. 
	
	Now consider any morphism $g\colon\overline{F}(X,E)\to\overline{F}(Y,G)\in\ca{F}$. Let $S_{*}\colon FX\looparrowright FY$ and $S^{*}\colon FY\looparrowright FX$ denote the relations corresponding to this morphism via the bijection of \ref{morphisms between quotients of congruences}, i.e. the relations $S_{*}=p_{(Y,G)}^{*}g_{*}p_{(X,E)*}$ and $S^{*}=p_{(X,E)}^{*}g^{*}p_{(Y,G)*}$. Now since $F$ is a full and covering regular functor, we know by the previous lemma that it is also full on subobjects and so there exist relations $R_{*}\colon X\looparrowright Y$ and $R^{*}\colon Y\looparrowright X$ in $\ca{E}$ such that $F(R_{*})=S_{*}$ and $F(R^{*})=S^{*}$. Furthermore, we have the following:
	\begin{equation*}
	F(GR_{*}E)=F(G)F(R_{*})F(E)=F(G)S_{*}F(E)=S_{*}=F(R_{*})
	\end{equation*}
	\begin{equation*}
	F(R^{*}R_{*})=F(R^{*})F(R_{*})=S^{*}S_{*}\supseteq F(E)
	\end{equation*}
	\begin{equation*}
	F(R_{*}R^{*})=F(R_{*})F(R^{*})=S_{*}S^{*}\subseteq F(G)
	\end{equation*}
	But now because $F$ is fully (order-) faithful it reflects inclusions of subobjects. Thus, we deduce that $GR_{*}E=R_{*}$, $R^{*}R_{*}\supseteq E$ and $R_{*}R^{*}\subseteq G$, so that $R_{*}$ is a morphism $(X,E)\to(Y,G)\in\ca{E}_{ex/reg}$.
	
	Finally, we have by definition of the functor $\overline{F}$ that
	\begin{align*}
	\overline{F}(R_{*})_{*} &=p_{(Y,G)*}F(R_{*})p_{(X,E)}^{*}=p_{(Y,G)*}S_{*}p_{(X,E)}^{*}=p_{(Y,G)*}p_{(Y,G)}^{*}g_{*}p_{(X,E)*}p_{(X,E)}^{*} \\
	&=I_{\overline{F}(Y,G)}g_{*}I_{\overline{F}(X,E)}=g_{*}
	\end{align*}
	and hence $\overline{F}(R_{*})=g$.
\end{proof}

The final ingredient we need is the $\nc{Pos}$-enriched analogue of Lemma 1.4.9 from \cite{Makkai-Reyes}.

\begin{lemma}
	Let $F\colon\ca{E}\to\ca{F}$ be a regular functor between regular categories where moreover $\ca{E}$ is exact. If $F$ is fully order-faithful and covering, then it is an equivalence.
\end{lemma}
\begin{proof}
	It suffices to show that $F$ is essentially surjective on objects. So let $Y\in\ca{F}$. By the assumption that $F$ is covering, we can find a coinserter $q\colon FX\twoheadrightarrow Y$ in $\ca{F}$ for some object $X\in\ca{E}$. Consider the kernel congruence $\xy\xymatrix{q/q\ar@<1ex>[r]^{q_{0}}\ar@<-1ex>[r]_{q_{1}} & FX}\endxy$ in $\ca{F}$. Since $F$ is a covering and full regular functor, we know that it must be also full on subobjects. In particular, there is a relation $\xy\xymatrix{E\ar@<1ex>[r]^{e_{0}}\ar@<-1ex>[r]_{e_{1}} & X}\endxy$ in $\ca{E}$ such that $F(E)=q/q$ as relations in $\ca{F}$.
	
	Now because $F$ is order-faithful, $F(E)=q/q$ being a congruence in $\ca{F}$ implies that $E$ is a congruence on $X$ in $\ca{E}$. Since $\ca{E}$ is assumed to be exact, $E$ has a coinserter, say $p\colon X\to P$, and is the kernel congruence of that coinserter. Now by regularity of the functor $F$ we have that $Fp\colon FX\twoheadrightarrow FP$ is the coinserter of $F(E)=q/q$. Since $q\colon FX\to Y$ is also a coinserter of $q/q$, we deduce that there exists an iso $FP\cong Y$.
\end{proof}

Now putting everything together we have proved the following result.

\begin{theorem}\label{Characterization of ex/reg}
	Let $F\colon\ca{E}\to\ca{F}$ be a regular functor with $\ca{F}$ an exact category. Then $\overline{F}\colon\ca{E}_{ex/reg}\to\ca{F}$ is an equivalence if and only if $F$ is fully order-faithful and covering.
\end{theorem}

In other words, if have found a regular functor $F\colon\ca{E}\to\ca{F}$ into an exact category $\ca{F}$ and $F$ is fully order-faithful and covering, then we have identified $\ca{F}$ as $\ca{E}_{ex/reg}$. This can be immediately applied to produce some examples of categories which are exact completions of regular categories.

\begin{example}
	\begin{enumerate}
		\item Consider $\nc{Set}$ viewed as $\nc{Pos}$-category with discrete hom-sets. We have seen that it is regular but not exact. The discrete poset functor $D\colon\nc{Set}\to\nc{Pos}$ is clearly fully order-faithful and regular. It is also trivially covering, since for any poset $(X,\leq)$ we have an order-preserving surjection $(X,=)\twoheadrightarrow(X,\leq)$. Thus, by \ref{Characterization of ex/reg} we have that $\nc{Set}_{ex/reg}\simeq\nc{Pos}$.
		\item Consider the category $\nc{OrdMon}$ of ordered monoids, which is a variety of ordered algebras in the sense of Bloom \& Wright\cite{Bloom & Wright} and hence is an exact category\cite{Kurz-Velebil}. Similarly, the category $\nc{OrdMon}_{can}$ of cancellative ordered monoids is regular, since it is an ordered quasivariety. Then by \ref{Characterization of ex/reg} we see that $(\nc{OrdMon}_{can})_{ex/reg}\simeq\nc{OrdMon}$. Indeed, every ordered monoid admits a surjective homomorphisms from a free one and clearly every free ordered monoid is cancellative. Recall here that the free ordered monoid $F(X,\leq)$ on a poset $(X,\leq)$ has elements all finite lists $(x_{1},x_{2},...,x_{n})$ of elements of $X$, with $(x_{1},x_{2},...,x_{n})\leq (y_{1},y_{2},...,y_{m})$ in $F(X,\leq)$ if and only if $m=n$ and $x_{i}\leq y_{i}$ for all $i\in\{1,2,...,n\}$.
		\item The ordinary category $\nc{Mon}$ of monoids is regular, hence also regular as a locally discrete $\nc{Pos}$-category. The inclusion functor $\nc{Mon}\hookrightarrow\nc{OrdMon}$ is regular and any ordered monoid $(M,\leq)$ admits a surjective homomorphism $(M,=)\twoheadrightarrow(M,\leq)$. It follows then from \ref{Characterization of ex/reg} that $\nc{Mon}_{ex/reg}\simeq\nc{OrdMon}$.
		\item It is easy to see that in the above example there is nothing special about the variety of monoids. Indeed, any ordinary quasivariety gives rise to a corresponding quasivariety of ordered algebras defined by the same set of axioms. The ordinary (unordered) version sits inside the ordered one as the discrete ordered algebras and as such is a regular subcategory. It follows then that its exact completion qua $\nc{Pos}$-category yields precisely the corresponding ordered quasivariety. Thus, for example, in the case of semigroups we similarly have $\nc{SGrp}_{ex/reg}\simeq\nc{OrdSGrp}$.
	\end{enumerate}
\end{example}

The examples presented so far of exact completions have all been varieties of ordered algebras which appear as completions of certain corresponding quasivarieties. However, the main example we would like to present in this section is order-topological in nature and involves the category of Priestley spaces. Let us thus first recall some terminology.

We will say that a triple $(X,\tau,\leq)$ with $\tau$ a topology and $\leq$ a partial order relation on the set $X$ is an \emph{ordered topological space}. 

A \emph{compact ordered space} is an ordered topological space $(X,\tau,\leq)$ such that $(X,\tau)$ is compact and $\leq$ is closed as a subspace of $X\times X$. This class of spaces was introduced and developed by L. Nachbin in \cite{Nachbin} as an ordered analogue of compact Hausdorff spaces, and so we will also call these \emph{Nachbin} spaces. Together with the continuous order-preserving functions between them they form a category which we denote by $\nc{Nach}$. Note that, under the assumption of compactness, the condition that the order relation be closed in the product space $X\times X$ means the following: whenever $x,y\in X$ with $x\nleq y$, there exist a open upper set $U$ and an open lower set $V$ such that $x\in U$, $y\in V$ and $U\cap V=\emptyset$.

Inside $\nc{Nach}$ sits the very interesting full subcategory $\nc{Pries}$ of \emph{Priestley} spaces. This is a class of ordered topological spaces introduced by H. A. Priestley\cite{Priestley} in order to provide an extension of Stone duality to distributive lattices. In other words there is an equivalence of categories $\nc{DLat}^{op}\simeq\nc{Pries}$, where $\nc{DLat}$ denotes the category of (bounded) distributive lattices and lattice homomorphisms.

Recall then than an ordered topological space $(X,\tau,\leq)$ is a \emph{Priestley space} if $(X,\tau)$ is compact and the following is satisfied: whenever $x\nleq y$, there exists a clopen upper set $U$ such that $x\in U$ and $y\notin U$. Ordered spaces satisfying the latter condition are often called \emph{totally order-separated}. It is immediate that every Priestley space is indeed a Nachbin space. It is furthermore clear that the underlying topological space of a Priestley space is a Stone space. In fact, the category $\nc{Stone}$ of Stone spaces is embedded in $\nc{Pries}$ as the full subcategory on the objects for which the order relation is discrete.

\begin{proposition}
	The category $\nc{Nach}$ is exact, while the category $\nc{Pries}$ is regular.
\end{proposition}
\begin{proof}
	Observe that the coinserter $\begin{tikzcd}E\ar[r,shift left=1ex]\ar[r,shift right=1ex] & X\ar[r,two heads,"q"] & X/E\cap E^{\circ}\end{tikzcd}$ of any internal congruence $E\rightarrowtail X\times X$ in $\nc{Nach}$ is constructed by equipping the set $X/E\cap E^{\circ}$ with the quotient topology and the induced order relation by the pre-order $E$. Since $E$ is closed in $X\times X$, so is the equivalence relation $E\cap E^{\circ}$ and so at the level of spaces we know that the quotient will be a compact Hausdorff space. It is then a Nachbin space because the order relation is by definition equal to $(q\times q)[E]$ and the map $q\times q$ is closed.
	
	It now follows that the effective epimorphisms in $\nc{Nach}$ are precisely the continuous monotone surjections. Indeed, if $f\colon X\to Y\in\nc{Nach}$ is surjective then on the level of spaces it is a continuous surjection between compact Hausdorff spaces and hence it is a quotient map. This means that the induced $\bar{f}\colon X/R\to Y$ is a homeomorphism, where $R=\{(x,x')| f(x)=f(x')\}=E\cap E^{\circ}$, for $E\coloneqq f/f$. But $\bar{f}$ also preserves and reflects the order because by definition we have $\bar{f}([x])\leq\bar{f}([x'])\iff f(x)\leq f(x')\iff (x,x')\in E\iff [x]\leq[x']$. Thus, $f$ is the coinserter of its kernel congruence.
	
	This shows that $\nc{Nach}$ is regular, since the continuous monotone surjections are clearly stable under pullback. To see that $\nc{Pries}$ is regular, it suffices to observe that the latter is closed under finite limits and subobjects in $\nc{Nach}$.
	
	Finally, consider any internal congruence $E\rightarrowtail X\times X$ in $\nc{Nach}$ and construct its coinserter $q$ as we did above. It is then immediate by the construction that $E=q/q$ and so we have proved that $\nc{Nach}$ is exact.
\end{proof}

We can now deduce an ordered version of the folklore result which identifies the exact completion of $\nc{Stone}$ as the category of compact Hausdorff spaces. The latter, to the best of the author's knowledge, seems to have first appeared in print in \cite{Panagis & me} where a similar argument involving the ordinary exact completion was invoked.

\begin{corollary}
	$\nc{Pries}_{ex/reg}\simeq\nc{Nach}\simeq\nc{Stone}_{ex/reg}$.
\end{corollary}
\begin{proof}
	The inclusions $\nc{Stone}\hookrightarrow\nc{Pries}\hookrightarrow\nc{Nach}$ are both regular functors. Furthermore, if $X\in\nc{Nach}$, then $X$ is in particular a compact Hausdorff space and so admits a continuous surjection $\beta(X)\twoheadrightarrow X$ from a Stone space $\beta(X)$, the latter being the Stone-Cech compactification of the discrete set $X$. Equipping $\beta(X)$ with the equality relation this becomes a continuous monotone surjection in $\nc{Nach}$. Thus, both inclusion functors are also covering and the result follows from \ref{Characterization of ex/reg}.
\end{proof}

Before ending this section, let us record a small observation that generalizes some of the examples we have seen so far. To this effect, recall that some varieties of ordered algebras were described as exact completions of certain ordinary varieties which appeared as the objects with discrete order relation. For example, we had $\nc{Mon}_{ex/reg}\simeq\nc{OrdMon}$ for the category of ordered monoids. Similarly, in the context of the above corollary we could have included the equivalence $\nc{CHaus}_{ex/reg}\simeq\nc{Nach}$, where $\nc{CHaus}$ denotes the locally discrete category of compact Hausdorff spaces.

More generally now, consider any regular category $\mathcal{E}$ and define an object $X\in\mathcal{E}$ to be \emph{discrete} if for every $f,g\colon A\to X\in\ca{E}$ we have that $f\leq g\implies f=g$. If we denote by $\nc{Dis}(\ca{E})$ the full subcategory on the discrete objects, then it is plain that $\nc{Dis}(\ca{E})$ is a locally discrete category which is closed under finite limits and subobjects in $\ca{E}$. Thus, $\nc{Dis}(\ca{E})$ is a regular category as well. We will say that $\ca{E}$ has \emph{enough discrete objects} if for every object $X\in\ca{E}$ there exists an $\nc{so}$-morphism $D\twoheadrightarrow X$ in $\ca{E}$ with $D\in\nc{Dis}(\ca{E})$. By another application of \ref{Characterization of ex/reg} we now deduce the following:

\begin{corollary}
	Let $\ca{E}$ be an exact category with enough discrete objects. Then, $\ca{E}\simeq\nc{Dis}(\ca{E})_{ex/reg}$.
\end{corollary}

\section{Internal Posets and Exact Completion}

In this final section we consider the process of taking internal posets in an ordinary category and how the ordinary and enriched notions of regularity and exactness are related through said process. Furthermore, we prove a type of commutation between this construction of internal posets and that of exact completion.

To begin with, suppose that $\ca{C}$ is any finitely complete ordinary category. We can then define a category $\nc{Ord}(\ca{C})$ as follows:
\begin{itemize}
	\item \underline{Objects:} are pairs $(X,\leq_{X})$, where $X$ is an object of $\mathcal{C}$ and $\leq_{X}\rightarrowtail X\times X$ is a partial order relation in $\ca{C}$.
	\item \underline{Morphisms:} A morphism $f\colon(X,\leq_{X})\to(Y,\leq_{Y})\in\nc{Ord}(\ca{C})$ is a morphism $f\colon X\to Y\in\ca{C}$ such that $f(\leq_{X})\subseteq\leq_{Y}$.
\end{itemize}
The condition $f(\leq_{X})\subseteq\leq_{Y}$ means that there is a commutative diagram in $\ca{C}$ of the form
\begin{displaymath}
	\begin{tikzcd}
		\leq_{X}\ar[d,tail]\ar[r,dashed] & \leq_{Y}\ar[d,tail] \\
		X\times X\ar[r,"f\times f"'] & Y\times Y
	\end{tikzcd}
\end{displaymath}
Composition of morphisms and identities are those of $\ca{C}$.

Furthermore, given morphisms $f,f'\colon(X,\leq_{X})\to(Y,\leq_{Y})\in\nc{Ord}(\ca{C})$, we define $f\leq f'$ to mean that there exists a commutative diagram
\begin{displaymath}
	\begin{tikzcd}
		X\ar[dr,"\langle f{,}g\rangle"']\ar[rr,dashed] & & \leq_{Y}\ar[dl,tail] \\
		& Y\times Y &
	\end{tikzcd}
\end{displaymath}
Now it is easy to see that this order relation on morphisms of $\nc{Ord}(\ca{C})$ is preserved by composition. For example, if $f,f'\colon(X,\leq_{X})\to(Y,\leq_{Y})\in\nc{Ord}(\ca{C})$ with $f\leq f'$ and $g\colon(Y,\leq_{Y})\to(Z,\leq_{Z})$, we have $gf\leq gf'$ by pasting the following commutative diagrams
\begin{displaymath}
	\begin{tikzcd}
		X\ar[dr,"\langle {f,f'}\rangle"']\ar[rr,dashed] & & \leq_{Y}\ar[dl,tail]\ar[rr,dashed] & & \leq_{Z}\ar[dl,tail] \\
		& Y\times Y\ar[rr,"g\times g"'] & &Z\times Z & 
	\end{tikzcd}
\end{displaymath}

Thus, $\nc{Ord}(\ca{C})$ is enriched in $\nc{Pos}$. Our first observation below is that finite completeness of the ordinary category $\ca{C}$ implies the existence of all finite weighted limits in $\nc{Ord}(\ca{C})$.

\begin{proposition}
	If $\ca{C}$ is finitely complete, then $\nc{Ord}(\ca{C})$ has finite weighted limits.
\end{proposition}
\begin{proof}
	It is easy to see that $\begin{tikzcd}(X,\leq_{X}) & (X\times Y,\leq_{X}\times\leq_{Y})\ar[l,"\pi_{X}"']\ar[r,"\pi_{Y}"] & (Y,\leq_{Y})\end{tikzcd}$ is a product diagram for every $(X,\leq_{X}),(Y,\leq_{Y})\in\nc{Ord}(\ca{C})$. Let us show how to construct the inserter of a pair of morphisms $\begin{tikzcd}(X,\leq_{X})\ar[r,shift left=.75ex,"f"]\ar[r,shift right=.75ex,"g"'] & (Y,\leq_{Y})\end{tikzcd}$. For this, form the following pullback square in $\ca{C}$
	\begin{displaymath}
		\begin{tikzcd}[sep=large]
			E\ar[d,tail,"e"']\ar[r,"e'"] & \leq_{Y}\ar[d,tail] \\
			X\ar[r,"\langle{f,g}\rangle"'] & Y\times Y
		\end{tikzcd}
	\end{displaymath}
	Let $\leq_{E}$ be the restriction of $\leq_{X}$ to the subobject $E\rightarrowtail X$, i.e. $\leq_{E}=(E\times E)\cap\leq_{X}$ as subobjects of $X\times X$. It is easy to see that $\leq_{E}$ is itself an internal partial order relation on $E\in\ca{C}$ so that we have a morphism $e\colon(E,\leq_{E})\to(X,\leq_{X})\in\nc{Ord}(\ca{C})$. Also, by commutativity of the pullback square above we have $fe\leq ge$.
	
	Now let $h\colon(Z,\leq_{Z})\to(X,\leq_{X})$ be such that $fh\leq gh$. This means that $\langle fh,gh\rangle=\langle f,g\rangle h$ factors through $\leq_{Y}\rightarrowtail Y\times Y$, say via $u\colon Z\to\leq_{Y}$, so then by the pullback property there exists a unique $v\colon Z\to E$ satisfying $ev=h$, $e'v=u$.
	
	Finally, $e$ is an $\nc{ff}$-morphism in $\nc{Ord}(\ca{C})$ by definition of $\leq_{E}$. Indeed, for any $\begin{tikzcd}(Z,\leq_{Z})\ar[r,shift left=.75ex,"h"]\ar[r,shift right=.75ex,"h'"'] & (E,\leq_{E})\end{tikzcd}$, the inequality $eh\leq eh'$ means that $(e\times e)\langle h,h'\rangle$ factors through $\leq_{X}$, which implies $\langle h,h'\rangle$ factors through $\leq_{E}=(E\times E)\cap\leq_{X}$, i.e. that $h\leq h'$.
\end{proof}
\vspace{3mm}

Since it will be needed later, let us also record here how to construct comma squares
\begin{displaymath}
	\begin{tikzcd}
		(C,\leq_{C})\ar[r,"c_{1}"]\ar[d,"c_{0}"']\ar[dr,phantom,"\leq"] & (Y,\leq_{Y})\ar[d,"g"] \\
		(X,\leq_{X})\ar[r,"f"'] & (Z,\leq_{Z})
	\end{tikzcd}
\end{displaymath}
in $\nc{Ord}(\ca{C})$. This is accomplished by constructing the following pullback square in $\ca{C}$
\begin{displaymath}
	\begin{tikzcd}
		C\ar[r]\ar[d,tail,"\langle{c_{0},c_{1}}\rangle"'] & \leq_{Z}\ar[d,tail] \\
		X\times Y\ar[r,"f\times g"'] & Z\times Z
	\end{tikzcd}
\end{displaymath}
and then setting $\leq_{C}\coloneqq(C\times C)\cap(\leq_{X}\times\leq_{Y})$.

Before moving on, let us also discuss $\nc{ff}$-morphisms in $\nc{Ord}(\ca{C})$. We saw in the course of the previous proof that an $m\colon(X,\leq_{X})\to(Y,\leq_{Y})\in\nc{Ord}(\ca{C})$ with $m\colon X\to Y\in\ca{C}$ monomorphic and $\leq_{X}=(X\times X)\cap\leq_{Y}$ is an $\nc{ff}$-morphism in $\nc{Ord}(\ca{C})$. It is in fact not too hard to see that this completely characterizes $\nc{ff}$-morphisms in $\nc{Ord}(\ca{C})$. Indeed, assume that $m\colon(X,\leq_{X})\to(Y,\leq_{Y})\in\nc{Ord}(\ca{C})$ is an $\nc{ff}$-morphism. If $f,g\colon Z\to X\in\ca{C}$ are such that $mf=mg$, then we can also consider them as morphisms $f,g\colon(Z,\Delta_{Z})\to(X,\leq_{X})$ in $\nc{Ord}(\ca{C})$ and hence deduce that $f=g$. This proves $m$ must be a monomorphism in $\ca{C}$. Now arguing with generalized elements one can easily deduce that $\leq_{X}$ is indeed the restriction of $\leq_{Y}$ along $m\colon X\rightarrowtail Y$.

\begin{lemma}
	If $f\colon(X,\leq_{X})\to(Y,\leq_{Y})\in\nc{Ord}(\ca{C})$ is such that the underlying $f\colon X\to Y$ is a strong epimorphism in $\ca{C}$, then $f$ is an $\nc{so}$-morphism in $\nc{Ord}(\ca{C})$.
\end{lemma}
\begin{proof}
	Consider the following commutative square in $\nc{Ord}(\ca{C})$, where we assume $m\colon(M,\leq_{M})\to(Z,\leq_{Z})$ is an $\nc{ff}$-morphism.
	\begin{displaymath}
		\begin{tikzcd}
			(X,\leq_{X})\ar[r,"f"]\ar[d,"h"'] & (Y,\leq_{Y})\ar[d,"g"] \\
			(M,\leq_{M})\ar[r,tail,"m"'] & (Z,\leq_{Z})
		\end{tikzcd}
	\end{displaymath}
	In particular, by the preceding discussion we know that $m$ is monomorphic in $\ca{C}$ and so by the property of $f$ as a strong epimorphism in the latter category we deduce the existence of a $u\colon Y\to M$ such that $uf=h$ and $mu=g$. The fact that $u$ is actually a morphism in $\nc{Ord}(\ca{C})$ follows because $mu=g$ is such a morphism and because $m$ being an $\nc{ff}$-morphism also means that $\leq_{M}=(M\times M)\cap\leq_{Z}$.
\end{proof}

We can now prove that ordinary regularity of $\ca{C}$ implies (enriched) regularity for $\nc{Ord}(\ca{C})$. Note that this result is in some sense a special case of Proposition 62 in \cite{2-dim reg & ex}, where the authors prove a form of 2-categorical regularity for the 2-category $\nc{Cat}(\ca{E})$ of internal categories in a regular (1-)category $\ca{E}$. Nevertheless, we include a proof here in order to make the paper more self-contained. 

\begin{proposition}
	If $\ca{C}$ is an ordinary regular category, then $\nc{Ord}(\ca{C})$ is regular.
\end{proposition}
\begin{proof}
	Consider any $f\colon(X,\leq_{X})\to(Y,\leq_{Y})\in\nc{Ord}(\ca{C})$ along with its (regular epi,mono) factorization $\begin{tikzcd}X\ar[r,two heads,"p"] & M\ar[r,tail,"m"] & Y\end{tikzcd}$ in $\ca{C}$. Then by what we have already established earlier, upon setting $\leq_{M}\coloneqq(M\times M)\cap\leq_{Y}$, we obtain an ($\nc{so}$,$\nc{ff}$) factorization $\begin{tikzcd}(X,\leq_{X})\ar[r,two heads,"p"] & (M,\leq_{M})\ar[r,tail,"m"] & (Y,\leq_{Y})\end{tikzcd}$ in $\nc{Ord}(\ca{C})$.
	
	Now the existence of these factorizations together with the previous lemma imply that any $f\colon(X,\leq_{X})\to(Y,\leq_{Y})\in\nc{Ord}(\ca{C})$ is an $\nc{so}$-morphism in $\nc{Ord}(\ca{C})$ if and only if $f\colon X\to Y$ is a regular(=strong) epimorphism in $\ca{C}$. For the ``only if'' direction, suppose that $f\colon(X,\leq_{X})\to(Y,\leq_{Y})$ is an $\nc{so}$-morphism and consider its factorization $\begin{tikzcd}(X,\leq_{X})\ar[r,two heads,"p"] & (M,\leq_{M})\ar[r,tail,"m"] & (Y,\leq_{Y})\end{tikzcd}$ as constructed above. Then we have that $m\colon(M,\leq_{M})\to (Y,\leq_{Y})$ is both an $\nc{so}$ and $\nc{ff}$-morphism, hence is an isomorphism in $\nc{Ord}(\ca{C})$. In particular, $m$ is an isomorphism in $\ca{C}$ and so $p$ being a strong epimorphism in $\ca{C}$ implies the same for $f=mp$. 
	
	Finally, since pullbacks in $\nc{Ord}(\ca{C})$ are constructed by simply taking the pullback of the underlying morphisms in $\ca{C}$, pullback-stability of regular epimorphisms in $\ca{C}$ implies pullback-stability of $\nc{so}$-morphisms in $\nc{Ord}(\ca{C})$. 
\end{proof}

Similarly, ordinary exactness of $\ca{C}$ implies $\nc{Pos}$-enriched exactness of $\nc{Ord}(\ca{C})$. Again, this is a special case of Proposition 63 in \cite{2-dim reg & ex}.

\begin{proposition}\label{OrdC exact}
	If $\ca{C}$ is an ordinary exact category, then $\nc{Ord}(\ca{C})$ is exact.
\end{proposition}
\begin{proof}
	Suppose that the relation $\begin{tikzcd}(E,\leq_{E})\ar[r,tail,"\langle{e_{0},e_{1}}\rangle"] & (X,\leq_{X})\times (X,\leq_{X})\end{tikzcd}$ is a congruence on $(X,\leq_{X})\in\nc{Ord}(\ca{C})$. Then we have a monomorphism $E\rightarrowtail X\times X\in\ca{C}$, i.e. a relation $E$ on $X$ in $\ca{C}$. Reflexivity and transitivity of the congruence in $\nc{Ord}(\ca{C})$ imply the same properties for the relation $E\colon X\looparrowright X$ in $\ca{C}$. Then $R\coloneqq E\cap E^{\circ}$ is an equivalence relation on $X$ in $\ca{C}$.
	
	Since $\ca{C}$ is exact, there exists an exact sequence $\begin{tikzcd}R\ar[r,shift left=.75ex]\ar[r,shift right=0.75] & X\ar[r,two heads,"q"] & Q\end{tikzcd}$ in $\ca{C}$, which means that $q$ is the coequalizer of $R$ and $R$ is the kernel pair of $q$. Let $\leq_{Q}\coloneqq q(E)$ be the image of $E$ along $q$ in $\ca{C}$. Note that in the calculus of relation in the ordinary regular category $\ca{C}$ we can write $q(E)=qEq^{\circ}$, while exactness of the sequence is equivalent to $q^{\circ}q=R$ and $qq^{\circ}=\Delta_{Q}$.
	
	Now observe that $\leq_{Q}$ is indeed an internal partial order relation in $\ca{C}$. It is reflexive because $E$ is so and $q$ is a regular epimorphism. For transitivity we argue as follows
	\begin{displaymath}
		\leq_{Q}\circ\leq_{Q}=qEq^{\circ}qEq^{\circ}=qEREq^{\circ}=qE(E\cap E^{\circ})Eq^{\circ}=qEq^{\circ}=\leq_{Q}
	\end{displaymath}
	Finally, for anti-symmetry we have 
	\begin{align*}
		\leq_{Q}\cap\leq_{Q}^{\circ} &=qq^{\circ}(\leq_{Q}\cap\leq_{Q}^{\circ})qq^{\circ}=q(q^{\circ}\leq_{Q}q\hspace{1mm}\cap\hspace{1mm} q^{\circ}\leq_{Q}^{\circ}q)q^{\circ} \\
		&=q(E\cap E^{\circ})q^{\circ}=qq^{\circ}qq^{\circ}=\Delta_{Q}
	\end{align*}

	Now we claim that $\begin{tikzcd}(E,\leq_{E})\ar[r,shift left=.75ex,"e_{0}"]\ar[r,shift right=.75ex,"e_{1}"'] & (X,\leq_{X})\end{tikzcd}$ is the kernel congruence of the morphism $q\colon(X,\leq_{X})\to(Q,\leq_{Q})$ in $\nc{Ord}(\ca{C})$. So let $g_{0},g_{1}\colon(Z,\leq_{Z})\to(X,\leq_{X})$ be such that $qg_{0}\leq qg_{1}$. In terms of generalized elements in $\ca{C}$ this means that $(qg_{0},qg_{1})\in_{Z}\leq_{Q}$, which in turn is equivalent to $(g_{0},g_{1})\in_{Z}q^{-1}(\leq_{Q})$. But now observe that
	\begin{displaymath}
		q^{-1}(\leq_{Q})=q^{-1}(q(E))=q^{\circ}qEq^{\circ}q=RER=(E\cap E^{\circ})E(E\cap E^{\circ})=E
	\end{displaymath}
	Thus, $(g_{0},g_{1})\in_{Z} E$ as desired.
\end{proof}
\vspace{3mm}

In particular, if $\ca{C}$ is an ordinary regular category and $\ca{C}_{oex/reg}$ is its ordinary exact completion, then $\nc{Ord}(\ca{C}_{oex/reg})$ is an exact $\nc{Pos}$-category. In the remainder of this paper we want to prove that the latter category is equivalent to the exact completion in the enriched sense of both $\nc{Ord}(\ca{C})$ and $\ca{C}$ itself.

Consider a regular functor $F\colon\ca{C}\to\ca{D}$ between ordinary regular categories. Since $F$ preserves internal partial order relations, we have an induced functor $\nc{Ord}(F)\colon\nc{Ord}(\ca{C})\to\nc{Ord}(\ca{D})$ defined on objects by $(X,\leq_{X})\mapsto(FX,F(\leq_{X}))$. By the construction of finite weighted limits in $\nc{Ord}(\ca{C})$, the fact that $F$ preserves finite limits implies the same for the enriched functor $\nc{Ord}(F)$. Similarly, $F$ preserving regular epimorphisms translates to the fact that $\nc{Ord}(F)$ preserves $\nc{so}$-morphisms. Thus, $\nc{Ord}(F)$ is a regular functor between regular $\nc{Pos}$-categories.

We now turn to discussing how the properties of $F$ being fully faithful and covering translate to properties of $\nc{Ord}(F)$.

\begin{lemma}
	Let $F\colon\ca{C}\to\ca{D}$ be an ordinary regular functor which is fully faithful. Then $\nc{Ord}(F)\colon\nc{Ord}(\ca{C})\to\nc{Ord}(\ca{D})$ is fully order-faithful.
\end{lemma}
\begin{proof}
	It is clear that $\nc{Ord}(F)$ is faithful, since its action on morphisms is that of $F$ itself. Since $\nc{Ord}(\ca{C})$ has finite limits and $\nc{Ord}(F)$ preserves them, this is equivalent to order-faithfulness.
	
	Now consider any $h\colon(FX,F(\leq_{X}))\to (FY,F(\leq_{Y}))$. By fullness of $F$, there exists an $f\colon X\to Y\in\ca{C}$ such that $Ff=h$. It suffices then to show that $f$ is order-preserving. For this, observe that $Ff=h$ being a morphism in $\nc{Ord}(\ca{D})$ means that there is a commutative diagram in $\ca{D}$ of the form
	\begin{displaymath}
		\begin{tikzcd}
			F(\leq_{X})\ar[d,tail]\ar[r] & F(\leq_{Y})\ar[d,tail] \\
			F(X\times X)\cong FX\times FX\ar[r,"Ff\times Ff"'] & FY\times FY\cong F(Y\times Y)
		\end{tikzcd}
	\end{displaymath}
	By full faithfulness of $F$ this is then reflected to a commutative diagram in $\ca{C}$ which exhibits $f$ as a morphism $(X,\leq_{X})\to(Y,\leq_{Y})$ in $\nc{Ord}(\ca{C})$.
\end{proof}

\begin{lemma}
	Let $F\colon\ca{C}\to\ca{D}$ be an ordinary regular functor which is covering. Then $\nc{Ord}(F)\colon\nc{Ord}(\ca{C})\to\nc{Ord}(\ca{D})$ is covering.
\end{lemma}
\begin{proof}
	Consider any object $(Y,\leq_{Y})\in\nc{Ord}(\ca{D})$. Since $F$ is covering, we can find a regular epimorphism $q\colon FX\twoheadrightarrow Y$ in $\ca{D}$. We can then consider $FX$ with the discrete order and so we have an object $(FX,\Delta_{FX})\in\nc{Ord}(\ca{D})$. Since $q$ is a regular epimorphism in $\ca{D}$, we have an $\nc{so}$-morphism $q\colon(FX,\Delta_{FX})\twoheadrightarrow (Y,\leq_{Y})$ in $\nc{Ord}(\ca{D})$. That is to say, we have an $\nc{so}$-morphism $q\colon\nc{Ord}(F)(X,\Delta_{X})\twoheadrightarrow(Y,\leq_{Y})$ and so $\nc{Ord}(F)$ is covering.
\end{proof}

Putting everything together we now obtain the main result of this section.

\begin{proposition}
	For any regular ordinary category $\ca{C}$ there is an equivalence of $\nc{Pos}$-categories $\nc{Ord}(\ca{C}_{oex/reg})\simeq\nc{Ord}(\ca{C})_{ex/reg}\simeq\ca{C}_{ex/reg}$, where $\ca{C}_{oex/reg}$ denotes the exact completion of $\ca{C}$ as an ordinary category.
\end{proposition}
\begin{proof}
	The ordinary regular functor $\Gamma\colon\ca{C}\to\ca{C}_{oex/reg}$ is fully faithful and covering. By the preceding lemmas we have then that $\nc{Ord}(\Gamma)\colon\nc{Ord}(\ca{C})\to\nc{Ord}(\ca{C}_{oex/reg})$ satisfies the same properties in the enriched sense. Furthermore, the category $\nc{Ord}(\ca{C}_{oex/reg})$ is exact by \ref{OrdC exact} and thus from \ref{Characterization of ex/reg} we deduce that $\nc{Ord}(\ca{C}_{oex/reg})\simeq\nc{Ord}(\ca{C})_{ex/reg}$.
	
	Similarly, the composite functor $\ca{C}\to\nc{Ord}(\ca{C})\to\nc{Ord}(\ca{C}_{oex/reg})$ is regular, fully faithful and covering, being a composition of two functors satisfying these properties. Again by \ref{Characterization of ex/reg} we conclude that $\ca{C}_{ex/reg}\simeq\nc{Ord}(\ca{C}_{oex/reg})$.
\end{proof}

\vspace{3mm}

\subsection*{Acknowledgements}

I would like to thank my PhD advisors, Marino Gran and Panagis Karazeris, for their guidance during the time in which the research contained in this article was conducted. I am also indebted to Pierre-Alain Jacqmin for reading an earlier version of this paper and providing invaluable comments and feedback which improved the quality of the work. I furthermore thank Christina Vasilakopoulou and Konstantinos Tsamis for useful conversations regarding the topics of this paper. Special thanks are due to the anonymous referee, whose useful comments improved the presentation. The present work was financially supported by the Conseil de Recherche of the Universit\'e Catholique de Louvain in the form of a ``Fonds Sp\'eciaux de Recherche'' grant, for which I express my sincere gratitude.

\end{document}